\documentclass[a4paper,12pt]{amsart}
\usepackage[utf8]{inputenc}
\usepackage{amssymb,amsmath}
\usepackage{amscd}
\usepackage[all]{xy}
\usepackage[onehalfspacing]{setspace}
\usepackage{graphicx}
\usepackage{url}  
\usepackage[a4paper, left=2cm, right=2cm, top=3.5cm, bottom=3cm]{geometry}
\usepackage{tikz-cd}
\usepackage{array}
\usepackage{verbatim}

\def\A{\mathbb A}
\def\C{\mathbb C}

\def\K{\mathbb K}

\def\P{\mathbb P}
\def\R{\mathbb R}
\def\Q{\mathbb Q}

\def\Z{\mathbb Z}
\def\AA{\mathcal A}
\def\BB{\mathcal B}
\def\CC{\mathcal C}
\def\DD{\mathcal D}
\def\EE{\mathcal E}
\def\FF{\mathcal F}
\def\HH{\mathcal H}

\def\OO{\mathcal O}

\def\TT{\mathcal T}
\def\XX{\mathcal X}

\def\chr{{\operatorname{char}}}
\def\Hom{{\operatorname{Hom}}}
\def\Ker{{\operatorname{Ker}}}
\def\Coker{{\operatorname{Coker}}}

\def\Tor{{\operatorname{Tor}}}
\def\Spec{{\operatorname{Spec}}}

\def\Supp{{\operatorname{Supp}}}

\def\CH{{\operatorname{CH}}}
\def\Cl{{\operatorname{Cl}}}

\def\Pic{{\operatorname{Pic}}}

\def\id{{\operatorname{id}}}

\def\Gr{{\operatorname{gr}}}

\def\Coh{{\operatorname{Coh}}}

\def\Sing{{\operatorname{Sing}}}

\def\rk{{\operatorname{rk}}}
\def\lcm{{\operatorname{lcm}}}
\def\SL{{\operatorname{SL}}}
\def\Sh{{\operatorname{Sh}}}

\def\dg{{\operatorname{dg}}}

\def\can{{\operatorname{PD}}}
\def\PD{{\operatorname{PD}}}
\def\Sch{{\operatorname{Sch}}}
\def\chr{{\operatorname{char}}}

\def\wh{\widehat}
\def\ol{\overline}

\def\uI{\underline{\mathbb I}}
\def\uF{\underline{\mathbb F}}
\def\cHom{\HH om}
\def\Dsg{\DD^{\mathrm{sg}}}
\def\Db{\DD^b}
\def\Dm{\DD^-}
\def\Dperf{\DD^{\mathrm{perf}}}
\def\Perf{\DD^{\mathrm{perf}}}

\def\Ksg{{\mathrm K}^{sg}}
\def\KKsg{{\mathbb K}^{sg}}
\def\Kt{{\mathrm{K}}}
\def\Gt{{\mathrm{G}}}
\def\GL{{\mathrm{GL}}}
\def\cdh{{\mathrm{cdh}}}

\theoremstyle{plain}
\newtheorem{dummy}{dummy}[section]
\newtheorem{theorem}[dummy]{Theorem}
\newtheorem{proposition}[dummy]{Proposition}
\newtheorem{lemma}[dummy]{Lemma}
\newtheorem{corollary}[dummy]{Corollary}
\newtheorem{example}[dummy]{Example}
\newtheorem{application}[dummy]{Application}
\newtheorem{definition}[dummy]{Definition}
\newtheorem{remark}[dummy]{Remark}
\numberwithin{equation}{section}

\def\bal{\begin{aligned}}
\def\eal{\end{aligned}}

\newcommand{\blue}[1]{\leavevmode{\color{blue}{#1}}}

\newcolumntype{P}[1]{>{\centering\arraybackslash}p{#1}}
\newcolumntype{M}[1]{>{\centering\arraybackslash}m{#1}}

\title[Singularity $\Kt$-theory]{$\Kt$-theory and the singularity category of\\ quotient singularities}
\author{NEBOJSA PAVIC AND EVGENY SHINDER}

\begin{document}

\begin{abstract} 
In this paper we study Schlichting's $\Kt$-theory groups of the Buchweitz-Orlov
singularity category $\Dsg(X)$ of a quasi-projective algebraic scheme $X/k$
with applications to Algebraic $\Kt$-theory.

We prove for isolated quotient singularities over an algebraically closed field of characteristic zero that $\Kt_0(\Dsg(X))$
is finite torsion, and that $\Kt_1(\Dsg(X)) = 0$.
One of the main applications is that 
algebraic varieties with isolated quotient singularities 
satisfy rational Poincar\'e duality on the level of the Grothendieck group;
this allows computing the Grothendieck group of such varieties in terms of their resolution
of singularities. 
Other applications concern the Grothendieck group of perfect complexes supported at a singular point
and topological filtration on the Grothendieck groups.
\end{abstract}

\maketitle

\tableofcontents

\section*{Introduction} 

In this paper we perform a systematic study of the Schlichting $\Kt$-theory groups 
of the dg-enhancement of the Buchweitz-Orlov singularity category $\Dsg(X) = \Db(X) / \Dperf(X)$;
we call the latter $\Kt$-theory groups the \emph{singularity $\Kt$-theory}.

Let $X/k$ be a quasi-projective scheme.
Let $\Kt_i(X) = \Kt_i(\Dperf(X))$ denote the Thomason-Trobaugh 
$\Kt$-theory of perfect complexes, which in the quasi-projective case
coincides with $\Kt$-theory of vector bundles on $X$, 
while $\Gt_i(X) = \Kt_i(\Db(X))$ is Quillen's $\Gt$-theory, that is $\Kt$-theory
of coherent sheaves.
By construction the singularity $\Kt$-theory groups $\Ksg_i(X)$ fit into an exact sequence
\[
\dots \to \Kt_i(X) \to \Gt_i(X) \to \Ksg_i(X) \to \Kt_{i-1}(X) \to \dots,
\]
for $i \ge 1$, finishing at  
\begin{equation}\label{eq:short-k-th}
\dots \to \Ksg_1(X) \to \Kt_0(X) \to \Gt_0(X) \to \Ksg_0(X) \to 0,
\end{equation}
but negative $\Kt$ groups can be taken into account as well, see Lemma \ref{lem:sing-K-th}.

A classical result going back to Serre is that if $X$ is regular, then $\Dperf(X) = \Db(X)$,
so the canonical maps 
$\Kt_i(X) \to \Gt_i(X)$ are isomorphisms for all $i$ and $\Ksg_i(X) = 0$. 
In general we may think of the singularity $\Kt$-theory groups $\Ksg_i(X)$ as a tool
for controlling the difference between $\Kt$-theory and $\Gt$-theory. 
This approach is essentially
a homological incarnation of Orlov's definition of the singularity category,
and explains the terms ``singularity category'' and ``singularity $\Kt$-theory''.

We develop the theory of singularity $\Kt$-theory, explaining its functoriality properties
and stating relevant exact sequences. Many of these properties follow directly from the work of Orlov
\cite{orlov-sing-1, orlov-sing-2} once one makes sure that the relevant triangulated functors 
are induced from dg enhancements.
For a similar perspective on studying homological invariants of the singularity category 
see \cite{tabuada1, tabuada2, gratz-stevenson},
and for an algebraic approach to $\Kt_0$ and $\Kt_1$ of the singularity category 
via MCM modules see \cite{Holm, Navkal, Flores}.

Let us motivate our study from several viewpoints, relating to earlier work in 
Algebraic $\Kt$-theory of singular varieties, 
and pointing out what singularity $\Kt$-theory has to offer in each case.
As a general rule our most interesting applications concentrate on isolated rational singularities,
including quotient and $ADE$ singularities.

\subsection*{1. Poincar\'e duality for quotient singularities.}

One of the main questions which motivated this work has been the following one.
If $X/k$ is a quasi-projective algebraic variety with quotient singularities, it is a natural
guess that canonical maps $\Kt_i(X) \to \Gt_i(X)$ are isomorphisms up to torsion; indeed this could be expected
as $X$ should be thought of as an analog of a $\Q$-manifold, while $\Kt_i(X) \to \Gt_i(X)$ may be thought as 
the Poincar\'e duality map; we use the notation $\PD: \Kt_i(X) \to \Gt_i(X)$ for this map.
In general however it is not true that $\PD$ is an isomorphism up to torsion for varieties
with quotient singularities.

Indeed, if either $X$ has \emph{nonisolated} quotient singularities or if
$i \ge 1$, then examples of Gubeladze \cite{gubeladze} (cf Example \ref{ex:toric}) 
and Srinivas \cite{srinivas3} (see Remark \ref{rem:K-1-huge})
respectively show that $\Kt_i(X) \to \Gt_i(X)$ is not an isomorphism, even after tensoring with $\Q$. 
The typical phenomenon is that $\Gt_i(X)$ are under control while $\Kt_i(X)$ become counterintuitive.
In both examples of Srinivas and Gubeladze $\Kt_i(X) \to \Gt_i(X)$ has a ``huge'' kernel.

One of our main results is that Poincar\'e duality does hold up to torsion for $i = 0$ in the isolated quotient
singularities case:

\begin{theorem}[See Theorem \ref{theorem:main-thm}]
Let $X$ be an $n$-dimensional quasi-projective variety over an algebraically closed field $k$ of characteristic zero.
Assume that $X$ has isolated quotient singularities with isotropy groups $G_i$, $i = 1, \dots, m$.
Then the map
\[
\can: \Kt_0(X) \to \Gt_0(X) 
\]
is injective, and its cokernel is a finite torsion group annihilated by
$\lcm(|G_1|,\ldots ,|G_m|)^{n-1}$.
\end{theorem}

We deduce the following Corollary of the Theorem, which often allows to conclude that
$\Kt_0(X)$ is finitely generated:

\begin{corollary}[See Theorem \ref{thm:res-of-sing-injective}]
Under the same assumptions as in the Theorem, for any resolution of singularities $\pi: Y \to X$ the pull-back $\pi^*: \Kt_0(X) \to \Kt_0(Y)$ is injective.
\end{corollary}

This is a strengthening of a result of Levine who proves the result in dimension up to three and
shows that $\pi^*$ has torsion kernel in general \cite{levine}.

Let us emphasize that there is in principle no easy way of controlling $\Kt_0(X)$ of singular varieties.
To illustrate our point, let us note that it is a well-known open question in $\Kt$-theory
of singular varieties, whether every weighted projective space 
$X = \P(a_0, \dots, a_n)$ has a finitely generated $\Kt_0(X)$
\cite[Acknowledgements]{gubeladze}, \cite[5.2.3]{Joshua-Krishna}; 
note that $\Gt_0(X)$ is finitely generated
and has rank $n+1$.

One important method of computing $\Kt$-theory of singular varieties has been 
developed in \cite{weibel-co-cdh, weibel-co-toric, weibel-co-negative-bass, weibel-co, weibel-co-vorst}
and consists in relating $\Kt$-theory to various sheaf cohomology groups. This method
has been applied to weighted projective spaces in \cite{massey} where 
$\Kt_0(\P(1,\dots,1,a))$ has been computed (it is isomorphic to the Grothendieck group
of a projective space of the same dimension, which is the answer one would expect).

Regarding weighted projective spaces, we can prove the following.
If $a_0, \dots, a_n$ are pairwise coprime,
so that $X = \P(a_0, \dots, a_n)$ has isolated quotient singularities,
then using the Theorem and the Corollary above we deduce that 
$\Kt_0(X)$ is a free abelian group of rank $n+1$
(Application \ref{appl:weighted-Pn}).

\medskip

The Theorem above follows using the exact sequence \eqref{eq:short-k-th} once we know that
for varieties with isolated quotient singularities over an algebraically closed field of characteristic zero
$\Ksg_0(X)$ is finite torsion (Proposition \ref{prop:G_0})
and $\Ksg_1(X) = 0$ (Corollary \ref{K_1^Sg=0}). 
In order to study the general case we first study the local case $\A^n/G$, where
a finite group $G$ acts on its linear representation. We study this local case in some detail
relying on tools such as equivariant $\Kt$-theory, equivariant Chow groups
and cdh topology
(Propositions \ref{prop:AnG-c1}, 
\ref{prop:K_0^Sg-and-order},
\ref{prop:gsprverdier},
\ref{prop:K_j^Sg}).

In contrast to the isolated quotient singularities case, the 
group $\Ksg_1(X)$ in general does not vanish for more general singularities, e.g. for rational isolated
singularities (Example \ref{ex:cubic-cone}) or non-isolated quotient singularities (Example \ref{ex:toric}), 
and the assumption that
$k$ is algebraically closed is necessary as well (Example \ref{ex:A1-any-field}).

\subsection*{2. Cohomology and homology algebraic cycles.}

The usual Chow groups have the functoriality property of Borel-Moore homology theory, and it has
been asked by Srinivas what is the correct definition of Chow cohomology of singular varieties \cite{srinivas-ICM}.
Taking insight from the intersection homology, it seems natural that in order to define such a theory
one needs to generalize both the algebraic cycles and the rational equivalence relation. 
For example the Chow group
$\CH_{\dim(X)-1}(X)$ (which coincides with $\Cl(X)$ when $X$ is normal)  is the group of ``homology divisors'', whereas
$\Pic(X)$ can be thought as the group of ``cohomology divisors''. Cohomology 
zero cycles have been introduced and studied in \cite{levine-weibel}.

Let us now take the sheaves rather than cycles perspective, and see what we can say then.
This approach is legitimate as one way to define Chow groups (up to torsion) is 
to take the associated graded groups
for the topological filtration on $\Gt_0(X)$, that is the filtration given by codimension of support.
See work of Gillet \cite{gillet} for a realization of this approach to 
Chow groups and intersection theory in the regular case
and Fulton \cite{fulton-cohomology} for the view-point on 
$\Kt_0(X) \otimes \Q$ as a variant for Chow-cohomology theory.

In fact the difference between homology and cohomology algebraic 
cycles in the singular case
seems to be of the same nature as between coherent sheaves and perfect complexes.
For example, if we ask the question: 
what makes a homology class cohomological, this can be interpreted
as the question about 
\[
\Ksg_0(X) = \Coker(\Kt_0(X) \to \Gt_0(X)). 
\]

We explain that there is an induced ``topological'' filtration by codimension of support 
on $\Ksg_0(X)$ and study its associated
graded groups $\Gr^i \Ksg_0(X) = F^i \Ksg_0(X) / F^{i+1} \Ksg_0(X)$; 
we think of these groups as \emph{obstructions for the codimension $i$
homology algebraic cycles to be cohomological}. 
This approach can be demonstrated in small dimension and codimension as follows.

\begin{theorem}(Proposition \ref{prop:K_0sg-filtr})
Let $X/k$ be a connected reduced quasi-projective scheme of pure dimension
over an algebraically closed field, then

\noindent(1) $\Gr^0 \Ksg_0(X) = \Z^{N-1}$, where $N$ is the number of irreducible components of $X$.

\noindent(2) If $X$ is irreducible and normal, then $\Gr^1 \Ksg_0(X) = \Cl(X) / \Pic(X)$.

\noindent(3) $\Gr^{\dim{X}} \Ksg_0(X) = 0$. 
\end{theorem}

In particular, if $X$ is an irreducible normal
surface, then $\Ksg_0(X)$ is concentrated in a single degree $1$ and
$\Ksg_0(X) \simeq \Cl(X)/\Pic(X)$.

A related and especially amusing phenomenon is that the well-known Kn\"orrer periodicity shifts the 
topological filtration by one (see Proposition \ref{prop:knorrer-shift});
this puts questions such as factoriality of threefolds and irreducibility of curves on equal footing
(Application \ref{appl:Knorrer-3}, Example \ref{ex:ADE-dim3}).

In the case of ordinary double points of arbitrary dimension, the only nontrivial
graded group for the topological filtration on $\Ksg_0(X)$
is in the middle codimension (Examples \ref{ex:A-ev}, \ref{ex:A-odd}). 
Returning to the relation between sheaves and algebraic cycles,
this predicts that all cycles in codimension up to half the dimension on varieties with ordinary double
points are ``cohomological''. For a very concrete example, note that normal varieties of dimension at least four
with ordinary double points are factorial, that is $\Pic(X) = \Cl(X)$.

In the example of an isolated quotient singularity, for instance, in the local case
the associated graded groups $\Gr^i \Ksg_0(\A^n/G)$ are typically nonzero in the range
$0 < i < n$ and are closely related to the $G$-equivariant Chow groups of a point \cite{edidingraham}.

\subsection*{3. Computing $\Kt_0(X \text{ on } \Sing(X))$: the Srinivas conjecture.}

Let $X$ be a quasi-projective variety with isolated singularities
over an algebraically closed field of characteristic zero.
In \cite{srinivas2} Srinivas introduced and studied the Grothendieck group of 
the exact category of coherent sheaves supported at the singular locus
and having finite projective dimension (i.e. perfect as complexes); 
for Cohen-Macaulay isolated singularities
this group is isomorphic to the Grothendieck group $\Kt_0(X \text{ on } \Sing(X))$ 
of the triangulated category of 
zero-dimensional perfect complexes supported at the singular points \cite[Proposition 2]{RS}.

There is a natural homomorphism 
\[
l: \Kt_0(X \text{ on } \Sing(X)) \to \Z^{\Sing(X)}
\]
induced by the length of the sheaf. 

We call the question whether $l$ is an isomorphism for isolated quotient singularities 
\emph{the Srinivas conjecture} (see \cite[Page 38]{srinivas2}). 
Levine has proved that $l$ is surjective 
for all isolated Cohen-Macaulay singularities \cite[Proposition 2.6]{levine}.
Furthermore, Levine proved that 
for isolated quotient singularities of dimension up to three in characterstic zero
$l$ is an isomorphism
\cite[Theorem 3.3]{levine}, and that it is an isomorphism up to torsion in general \cite[Theorem 2.7]{levine}.

On the other hand, it is known that $l$ is not always injective;
for instance for a three-dimensional quadric cone $xy = zw$, 
$\Ker(l) = \Z \oplus k^*$ \cite[Theorem 4.2]{levine}.

Using singularity $\Kt$-theory with supports we reprove surjectivity of $l$ and
deal with its injectivity.
Namely, we prove that $l$ is injective for isolated quotient singularities over an algebraically closed field of characteristic zero (see Proposition \ref{prop:Srinivas-quotient}); 
this is a direct consequence of the fact that $\Ksg_1(X) = 0$ for such singularities.
We also note that in our approach the surjectivity of $l$
follows from the topological filtration considerations in subsection 2 and illustrates
the interaction between homology and cohomology cycles in dimension zero: skyscraper sheaf
of a singular point (homology cycle) is represented by a class of a perfect complex (cohomology cycle).

In general we show that any example where $\Kt_0(X) \to \Gt_0(X)$ has nonzero kernel will automatically
have $\Ker(l) \ne 0$ (Remark \ref{rem:ker-l}).

\subsection*{4. Homological Bondal-Orlov localization conjecture.}

Given a variety $X$ with rational singularities and $\pi: Y \to X$ a resolution of singularities,
it is a natural question whether $\pi_*: \Db(Y) \to \Db(X)$ is essentially surjective and
whether there is an equivalence $\Db(Y) / \Ker(\pi_*) \simeq \Db(X)$, that is $\Db(X)$ is
a Verdier quotient of $\Db(Y)$; this question may be called 
the \emph{Bondal-Orlov localization conjecture} \cite{bondal-orlov-icm}.
As a side result, which is morally related,
in a certain sense dual to, but
not dependent on the singularity category, we prove that Bondal-Orlov localization
conjecture holds for quotient singularities (not necessarily isolated)
in characteristic zero (Theorem \ref{thm:res-of-sing-ess-surjective}). This implies 
in particular that $\pi_*: \Gt_0(Y) \to \Gt_0(X)$ is surjective, which is the ``dual'' 
statement to the injectivity $\pi^*: \Kt_0(X) \to \Kt_0(Y)$ for isolated quotient singularities
explained in subsection 1 above (but there is no logical link between the two statements).

In the more general setting, to the best of our knowledge it is not known whether the pushforward
morphism
\[
\pi_*: \Gt_0(Y) \to \Gt_0(X) 
\]
is surjective if $X$ is a variety with rational singularities over an algebraically closed field
and $\pi: Y \to X$ is a resolution.
We call this question, as well as the long $\Gt$-theory exact sequence
\[
\dots \to \Kt_i(\Ker(\pi_*)) \to \Gt_i(Y) \to \Gt_i(X) \to \dots \to \Kt_0(\Ker(\pi_*)) \to \Gt_0(Y) \to \Gt_0(X) \to 0
\]
predicted by the Bondal-Orlov conjecture the \emph{Homological Bondal-Orlov conjecture}
and we hope return to this question in the future.

\subsection*{Acknowledgements}

We would like to thank A. Betina, T. Bridgeland, J. Greenlees, A. Efimov, M. Kalck, J. Karmazyn, J. Kass, A. Kuznetsov, 
N. Pagani, D. Pomerleano, R. Potter, M. Schlichting,
P. Sechin, V. Srinivas, P. Stellari, 
G. Stevenson, B. Totaro, A. Vishik, V. Vologodsky and M. Wemyss for 
helpful discussions and e-mail communication.
E.S. was partially supported by Laboratory of Mirror Symmetry NRU HSE, RF government grant, ag. N~14.641.31.0001.

\subsection*{Notation and conventions}

Unless specified otherwise, the schemes we consider are quasi-projective over a field $k$,
however most results remain true in the generality of Orlov's (ELF) condition \cite{orlov-sing-1}.
Furthermore, the base field $k$ is assumed to have characteristic zero; however all general results in Section \ref{sec:singularity-cat} are true without this assumption.

If $G$ is a finite group, we write $\wh{G}$ for the group of characters $\Hom(G, k^*)$.
We write $\Z_n$ for the cyclic group of order $n$.

All triangulated and dg categories are assumed to be $k$-linear.
All functors such as pull-back $\pi^*$, pushforward $\pi_*$ and tensor product $\otimes$ 
when considered between derived categories are derived functors.

\section{Singularity $\Kt$-theory}

\label{sec:singularity-cat}

\subsection{Triangulated and dg singularity categories}

We start by introducing the category whose $\Kt$-theory we are going to study. Unless stated otherwise, $X$ is a quasi-projective
scheme over a field $k$.
We write $\Db(X)$ for the bounded derived category of coherent sheaves on $X$
and $\Dperf(X)$ for its subcategory of perfect complexes, which in the quasi-projective case
coincides with the subcategory of bounded complexes of locally free sheaves.

\begin{definition}[Buchweitz \cite{buchweitz}, Orlov \cite{orlov-sing-1}] The triangulated category of singularities of $X$ is the Verdier quotient
$$\Dsg(X):=\Db(X)/\Perf(X).$$
\end{definition}

As we will be interested in $\Kt$-theory of the singularity category, 
we need to specify a dg-enhancement for $\Dsg(X)$ to apply Schlichting's machinery of $\Kt$-theory
of dg-categories.

For that we first recall that $\Db(X)$ has a unique dg-enhancement, up to quasi-equivalence 
\cite{lunts-orlov, canonaco-stellari}.
We denote this dg-enhancement by $\Db_{dg}(X)$.
Considering the full dg-subcategory of perfect objects in $\Db_{dg}(X)$, we
get a dg-enhancement $\Dperf_{dg}(X)$ of $\Dperf(X)$. Finally, applying the Drinfeld quotient 
construction \cite{drinfeld} to the pair $\Dperf_{dg}(X) \subset \Db_{dg}(X)$ 
we get a dg-category $\Dsg_{dg}(X)$ which is a dg-enhancement for $\Dsg(X)$.

We note that even though the dg-enhancement of $\Dsg(X)$ may not be unique, our
choice is canonical in a way that all enhancements of $\Dsg(X)$ induced by an enhancement of
$\Db(X)$ are quasi-equivalent.

Similarly, we consider the singularity category with supports. For any closed $Z \subset X$ let
\[
\Dsg_Z(X) = \Db_Z(X)/\Dperf_Z(X).
\]
Here $\Db_Z(X)$ consists of complexes in $\Db(X)$ acyclic away from $Z$, and 
$\Dperf_Z(X) = \Db_Z(X) \cap \Dperf(X)$.
A dg-enhancement of $\Db(X)$ induces one for $\Db_Z(X)$, and using the Drinfeld quotient construction,
$\Dsg_Z(X)$ acquires a dg-enhancement $\Dsg_{Z,dg}(X)$.

We now list some properties of the singularity categories, which are due to Orlov.
Even though Orlov formulates these results on the triangulated level, they all lift to the dg-enhancements
due to the fact that all well-defined derived pull-back and push-forward functors 
lift to dg-enhancements of $\Db(X)$ and $\Perf(X)$ \cite{schnurer}.

\begin{proposition}[Orlov \cite{orlov-sing-1}]\label{prop:reduction-open-containing-sing}
\label{prop:orl-restr}
Let $j: U\subset X$ be an open embedding such that $\Sing(X)\subset U$.
Then
$$j^*:\Dsg(X)\xrightarrow{\sim} \Dsg(U)$$
is an equivalence, induced by a functor between dg-enhancements.
\end{proposition}

\begin{theorem}[Kn\"orrer periodicity, Orlov \cite{orlov-sing-1}]\label{Thm:Knorrer} 
Let $X/k$ be a smooth quasi-projective scheme and let $f:X\to\A^1$ be a non-zero morphism. 
Define $g=f+xy:X\times\A^2\to\A^1$.
Let $Z_f=f^{-1}(\{0\})$ and $Z_g=g^{-1}(\left\lbrace 0\right\rbrace )$, 
and let $W=Z_f\times \{ 0 \} \times \A^1\subset X\times\A^2$. 
Furthermore, denote by $i:W\hookrightarrow Z_g$ the inclusion and $p:W\to Z_f$ the flat projection. Then
$$ i_*p^*:\Dsg(Z_f)\to \Dsg(Z_g)$$
is an equivalence of triangulated categories induced by a functor between dg-enhancements.
\end{theorem} 

Recall that a full triangulated subcategory $\TT$ of a triangulated
category $\DD$ is called dense if any object in $\DD$ is a direct summand of an object in $\TT$.

\begin{proposition}[Orlov \cite{orlov-sing-2}]\label{prop:orl-supp}
For any closed subset $Z\subseteq X$, the induced functor
$$\Dsg_Z(X) \to \Dsg(X)$$
is fully faithful and is induced by a functor between dg-enhancements.
Furthermore, if $\Sing(X) \subset Z$, then this functor has a dense image.
\end{proposition}

If $\DD$ is a triangulated category, then we write $\ol{\DD}$
for the idempotent completion
of $\DD$; $\ol{\DD}$
is a triangulated category by \cite{balmer-schlichting}.
If $\DD$ admits a dg-enhancement, then the fully faithful functor $\DD \subset \ol{\DD}$
is induced by a dg-functor (see e.g. \cite[1.6.2]{beil-vol}).

\begin{theorem}[Orlov \cite{orlov-sing-2}]\label{theorem:idemp-formal-completion}
Assume that the formal completions of $X$ and $X'$ along $\Sing(X)$ and $\Sing(X')$ respectively
are isomorphic. 
Then we have equivalences
\[
\Dsg_{\Sing(X)}(X) \simeq \Dsg_{\Sing(X')}(X')
\]
and
\[
\ol{\Dsg(X)} \simeq \ol{\Dsg(X')}
\]
induced by functors between dg-enhancements.
\end{theorem}

In light of our interest in idempotent completions we will also need the following
celebrated result by Thomason.

\begin{theorem}[Thomason, Theorem 2.1 in \cite{thomason}]\label{thm:thomason} 
Let $\mathcal{D}$ be an essentially small triangulated category, then there is a one-to-one correspondence
$$\lbrace\mathcal{T}\subseteq\mathcal{D}\mid\mathcal{T}\textup{ dense strictly full triang. subcat.}\rbrace\stackrel{\mathrm{1:1}}{\longleftrightarrow}\lbrace H\subseteq \Kt_0(\mathcal{D})\mid H\textup{ subgroup}\rbrace.$$
The correspondence sends strictly full dense subcategories $\mathcal{T}\subseteq\mathcal{D}$ to the image of $\Kt_0(\mathcal{T})$ in $\Kt_0(\mathcal{D})$ and the inverse sends a subgroup $H$ of $\Kt_0(\mathcal{D})$ to the full triangulated subcategory $\mathcal{D}_H$, where $\mathcal{D}_H:=\lbrace A\in\mathcal{D}\mid [A]\in H\subseteq \Kt_0(\mathcal{D})\rbrace$.
\end{theorem}

\subsection{$\Kt$-Theory of the singularity category}

Schlichting's construction of the $\Kt$-theory 
spectrum \cite{schlichting, schlichting-2} can be applied
to produce $\Kt$-theory groups $\K_i(\CC)$, $i \in \Z$ for
a $k$-linear pretriangulated dg-category $\CC$. 

The $\K_i$ groups are covariantly functorial for dg-functors;
we summarize their properties as follows. 
For a pretriangulated dg-category $\CC$
we write $H^0(\CC)$ for its triangulated homotopy category.

\begin{enumerate}
 \item[(0)] $\K_0(\CC)$ is the Grothendieck group of the idempotent completion of $H^0(\CC)$.
 \item If $\CC \to \CC'$ induces a fully faithful embedding $H^0(\CC) \to H^0(\CC')$
 with a dense image, then all $\K_i(\CC) \to \K_i(\CC')$ are isomorphisms. 
 \item If $\AA \to \BB \to \CC$ induces a fully faithful embedding 
 $H^0(\AA)\to H^0(\BB)$ such that $H^0(\AA)$ is the kernel of $H^0(\BB)\to H^0(\CC)$ and a
 fully faithful functor $H^0(\BB)/H^0(\AA) \to H^0(\CC)$
 with a dense image, then there is a long exact sequence 
 \[
   \dots \to \K_i(\AA) \to \K_i(\BB) \to \K_i(\CC) \to \K_{i-1}(\AA) \to \dots
 \]
 \item $\K_i(\Db_{dg}(X))$ are isomorphic to $\Gt_i(X)$, Quillen's $\Gt$-theory, that is $\Kt$-theory of coherent sheaves.
 In particular, $\K_i(\Db(X)) = 0$ for $i < 0$.
 \item $\K_i(\Dperf_{dg}(X))$ are isomorphic to $\Kt_i(X)$, the Thomason-Trobaugh $\Kt$-theory \cite{tt},
 which under our assumptions on $X$ (quasi-projective scheme over a field)
 are isomorphic to Quillen's $\Kt$-theory of vector bundles.
\end{enumerate}

\begin{remark}
It is well-known that $\Kt$-theory of triangulated categories satisfying the axioms analogous
to those listed above cannot be defined \cite{schlichting-example}; the counterexample
is provided by the two singularity categories of schemes 
which are equivalent as triangulated categories but are forced to have non-isomorphic higher $\Kt$-groups
if one assumes the long exact $\Kt$-theory sequence for Verdier quotients of triangulated categories.
\end{remark}

\begin{definition}
We define the singularity $\Kt$-theory groups of $X$ by
\[
\Ksg_i(X) = \left\{\begin{array}{cc}
                  \K_i(\Dsg_{dg}(X)), & i \ne 0 \\
                  \Kt_0(\Dsg_{dg}(X)), & i = 0 
                \end{array}
\right. 
\]
and we also define $\KKsg_0(X) = \K_0(\Dsg(X))$.
\end{definition}

\begin{remark}\label{rem:KK}
 We make a special consideration for $i = 0$ since by property (0) above
$\KKsg_0(X)$ is in fact the Grothendieck group of the idempotent completion of
$\ol{\Dsg(X)}$, and not of $\Dsg(X)$ itself. 
By Theorem \ref{thm:thomason} we have $\Ksg_0(X) \subset \KKsg_0(X)$.
On the other hand by property (1) it is true that for $i \ne 0$, $\Ksg_i(X) = \K_i({\Dsg_{dg}(X)})
\simeq \K_i(\ol{\Dsg_{dg}(X)})$.
\end{remark}

Let us write $\can: \Kt_i(X) \to \Gt_i(X)$ for the canonical ``Poincar\'e duality''
morphism induced by 
$\Dperf_{dg}(X) \subset \Db_{dg}(X)$.
Our main motivation in defining the singularity $\Kt$-theory is for studying this map.

\begin{lemma}[Singularity $\Kt$-theory exact sequences]\label{lem:sing-K-th}
We have exact sequences 
\begin{equation}\label{eq:seq-k}
   \dots \to \Kt_i(X) \overset{\can}{\to} \Gt_i(X) \to \Ksg_i(X) \to \dots \to \Kt_0(X) 
   \overset{\can}{\to} \Gt_0(X) \to \Ksg_0(X) \to 0,
\end{equation}
\begin{equation}\label{eq:k_-1}
0 \to \Ksg_0(X) \to \KKsg_0(X) \to \Kt_{-1}(X) \to 0.
\end{equation}
and isomorphisms for $j \ge 1$
\begin{equation}\label{eq:k_-j}
\Ksg_{-j}(X) \simeq \Kt_{-j-1}(X).
\end{equation}
\end{lemma}
\begin{proof}
The statement follows from a single $\Kt$-theory sequence using the properties
of Schlichting $\Kt$-groups given above and
the fact that the image of $\Gt_0(X)$ in $\KKsg_0(X)$ is $\Ksg_0(X)$.
\end{proof}

We record the following well-known result:

\begin{lemma}\label{lemma:idempotent-K-1}
$\Dsg(X)$ is idempotent complete if and only if $\Kt_{-1}(X) = 0$.
\end{lemma}
\begin{proof}
Using \eqref{eq:k_-1} we see that
vanishing of $\Kt_{-1}(X)$ is equivalent to $\Kt_0(\Dsg(X)) = \Kt_0(\ol{\Dsg(X)})$
which implies $\Dsg(X) = \ol{\Dsg(X)}$ by the theorem of Thomason 
(Theorem $\ref{thm:thomason}$). 
\end{proof}

Similarly to the definition of the singularity $\Kt$-theory, for every closed $Z \subset X$
we consider the singularity $\Kt$-theory with supports defined by
\[
\Ksg_i(X \text{ on } Z) = \left\{\begin{array}{cc}
                  \K_i(\Dsg_{Z,dg}(X)), & i \ne 0 \\
                  \Kt_0(\Dsg_{Z,dg}(X)), & i = 0 
                \end{array}
\right. 
\]

\begin{lemma}\label{lem:K-th-supp}
If $\Sing(X) \subset Z$, then we have natural isomorphisms
$\Ksg_i(X \text{ on } Z) \simeq \Ksg_i(X)$ for $i \ne 0$
and $\Ksg_0(X \text{ on } Z) \to \Ksg_0(X)$ is injective.
\end{lemma}
\begin{proof}
Follows from Proposition \ref{prop:orl-supp} and property (1) of Schlichting's $\Kt$-theory.
\end{proof}

\begin{remark}\label{rem:K-th-supp-seq}
There are exact sequences for singularity $\Kt$-theory with supports analogous to \eqref{eq:seq-k}, 
\eqref{eq:k_-1};
note that $\Kt_i(\Dperf_{dg,Z}(X)) = \Kt_i(X \text{ on } Z)$ are the Thomason-Trobaugh $\Kt$-theory groups
\cite{tt} while $\Kt_i(\Db_{dg,Z}(X)) \simeq \Gt_i(Z)$ is Quillen's $\Gt$-theory \cite{quillen}.
\end{remark}

We now discuss functoriality properties of $\Ksg_i$.

\begin{lemma}\label{lem:funct}
$\Ksg_i$ are contravariantly functorial for morphisms of finite $\Tor$-dimension, and are covariantly functorial for proper morphisms of finite 
$\Tor$-dimension.
\end{lemma}
\begin{proof}
This holds because of the triangulated singularity categories
have this functoriality \cite{orlov-sing-1},
and since the pull-back and pushforward functors are induced by dg-enhancements \cite{schnurer}.
\end{proof}

\begin{lemma}\label{lem:open-iso}
Let $U$ be an open subscheme of $X$ containing the singular locus $\Sing(X)$.
Then the inclusion $j: U \to X$ induces an isomorphism
$$j^*: \Ksg_i(X)\simeq \Ksg_i(U)$$
for all $i \in \Z$.
\end{lemma}
\begin{proof}
By Proposition \ref{prop:orl-restr} $j^*$ is a quasi-equivalence on dg singularity categories, so
it must induce an isomorphism on $\Kt$-theory groups.
\end{proof}

\begin{lemma}\label{lem:homotopy-inv}
Let $X$ be a quasi-projective variety, and let $p: V \to X$ be a vector bundle over $X$.
Then we have an isomorphism
$$p^*: \Ksg_0(X)\simeq \Ksg_0(V).$$
\end{lemma}

\begin{proof} Let $i: X \to V$ the zero section.
Since $p$ is flat and $i$ is a regular embedding,
both morphisms $p$ and $i$ are of finite $\Tor$ dimension,
so by Lemma \ref{lem:funct} we have pull-back homomorphisms
$p^*:\Ksg_i(X)\to \Ksg_i(V)$ and $i^*:\Ksg_i(V)\to \Ksg_i(X)$, 
and $i^*$ is left-inverse to $p^*$, in particular $p^*$ is injective.

On the other hand, from the diagram 
\begin{equation*}
\xymatrix{
    \Gt_0(X)\ar@{->}[d]_{p^*}^{\simeq}\ar@{->>}[r] & \Ksg_0(X)\ar@{->}[d]^{p^*}\\
 \Gt_0(V)\ar@{->>}[r] & \Ksg_0(V)}
\end{equation*}
we see immediately that $p^*:\Ksg_0(X)\to \Ksg_0(V)$ is surjective as well.
\end{proof}

\begin{remark}
The functors $\Ksg_i$ are not homotopy invariant for $i \ne 0$ in general. Consider for example
the case of $\Ksg_1$; if we have $\Ksg_1(X \times \A^1) \simeq \Ksg_1(X)$, then
using the five-lemma applied to the five bottom terms of the sequence \eqref{eq:seq-k}, 
we would deduce that $\Kt_0(X \times \A^1) \simeq \Kt_0(X)$ which typically does not hold for singular varieties.
\end{remark}

We will now present a method to compute $\Ksg_j(X)$ for a special
class of schemes which we call $\A^1$-contractible.
This approach generalizes the so-called Swan-Weibel homotopy trick, which is used to show that normal graded domains
have vanishing Picard group \cite[Lemma 5.1]{Murthy}.

\begin{definition}
We say that $X$ is \emph{$\A^1$-contractible}, if
there exists a morphism $H: X \times \A^1 \to X$ 
such that $H|_{X \times 1}$ is the identity map and 
$H|_{X \times 0}$ is a constant rational point $x_0 \in X$.
We also say that $H$ is a \emph{contraction} of $X$.
\end{definition}

\begin{lemma}\label{lem:contractible}
The following affine schemes are $\A^1$-contractible:
\begin{enumerate}
    \item $\A^n/G$, where $G$ acts linearly on $\A^n$
    \item $V(f) \subset \A^n$, where $f \in k[x_1, \dots, x_n]$
    is a weighted homogeneous polynomial
\end{enumerate}
\end{lemma}
\begin{proof}
(1) $\A^n$ admits a $G$-equivariant contraction $H_{\A^n}: \A^n \times \A^1 \to \A^n$, $H(v, t) = tv$
($G$ acts trivially on the $\A^1$ factor), which induces a contraction $H: \A^n/G \times \A^1 \to \A^n/G$. 

(2) By assumption the algebra $k[V(f)] = k[x_1, \dots, x_n]/(f)$ is positively graded. Let $w_i > 0$ be the weight of $x_i$.
Then the $k$-algebra morphism
\[
k[X] \to k[X,t], \quad x_i \mapsto t^{w_i} \cdot x_i 
\]
is well-defined and induces a contraction for $V(f)$. 
\end{proof}

\begin{proposition}\label{prop:A1-Ksg}
Let $X$ be $\A^1$-contractible.

(1) For every $j \ge 0$ the canonical map
$\can: \Kt_j(X) \to \Gt_j(X)$ 
factors as a composition $\Kt_j(X) \overset{x_0^*}{\to} \Kt_j(\Spec(k)) \overset{p^*}\to \Gt_j(X)$,
where $p: X \to \Spec(k)$ is the structure morphism.

(2) There is a natural isomorphism
\[
\Ksg_0(X) \simeq \Gt_0(X) / (\Z \cdot [\OO_X]).
\]

(3) If $X$ has a smooth rational point $x_1 \in X$, then
\[
\Gt_0(X) = \Z \cdot [\OO_X] \oplus \Ksg_0(X) 
\]
and for every $j \ge 1$ there is a short exact sequence
\[
0 \to \Coker(\Kt_j(k) \overset{p^*}{\to} \Gt_j(X)) \to \Ksg_j(X) \to 
\Ker(\Kt_{j-1}(X) \overset{x_0^*}{\to} \Kt_{j-1}(k)) \to 0.
\]
\end{proposition}
\begin{proof}
(1) The proof relies on the fact that the canonical map $\can$ commutes with pull-backs of finite 
$\Tor$ dimension as well as on homotopy invariance of $\Gt$-theory.
Let us write $i_0$, $i_1$ for the two embeddings of $X$ into $X \times \A^1$
corresponding to $0, 1 \in \A^1$. These embeddings define Cartier divisors, in particular
are regular, hence of finite $\Tor$-dimension.

In the computation below we use the notation $K(f)$ and $G(f)$ for the pull-backs 
on $\Kt$ and $\Gt$-theory respectively, and $\can_{Z}$ 
for
the canonical Poincar\'e duality map $\Kt_0(Z) \to \Gt_0(Z)$
on any $Z$.
We compute:
\[\bal
\can_X &= \can_X \circ K(i_1) \circ K(H) \\
&= G(i_1) \circ \can_{X \times \A^1} \circ K(H) \\
&= G(i_0) \circ \can_{X \times \A^1} \circ K(H) \\
&= \can_X \circ K(i_0) \circ K(H) \\
&= \can_X \circ K(p) \circ K(x_0) \\
&= G(p) \circ \can_{\Spec(k)} \circ K(x_0) \\
\eal\]
which is what we had to establish.

Let us now compute $\Ker(\can)$, $\Coker(\can)$ for $\can: \Kt_j(X) \to \Gt_j(X)$ using (1).
The map $x_0^*$ is always surjective (since $x_0^* \circ p^* = \id_{\Kt_j(k)}$), 
hence $\Coker(\can) = \Coker(p^*)$, and applying this to $j = 0$ using the $\Kt$-theory short exact sequence
\eqref{eq:seq-k}
we get (2).

If in addition $X$ admits a smooth rational point $x_1$, then the map $p^*$ is injective 
(since $x_1^* \circ p^* = \id_{\Kt_j(k)}$ for
the pull-back $x_1^*: \Gt_j(X) \to \Kt_j(k)$ for the regular embedding of $x_1$ into $X$), 
hence (1) implies $\Ker(\can) = \Ker(x_0^*)$.

Once we have identified the kernel and cokernel of $\can$, 
(3) follows from the $\Kt$-theory long exact sequence \eqref{eq:seq-k}.
\end{proof}

In the two examples below we consider $\A^1$-contractible schemes with no smooth rational points.

\begin{example}\label{ex:k-eps}
$X = \Spec(k[\epsilon] / \epsilon^n)$ is $\A^1$-contractible 
by Lemma \ref{lem:contractible} (2).
In this case the canonical map $\Kt_0(X) \to \Gt_0(X)$ is $\Z \overset{\cdot n}{\to} \Z$
and $\Ksg_0(X) = \Z_n$.
\end{example}

\begin{example}\label{ex:x2-y2}
Let $X$ be an affine curve $x^2 + y^2 = 0$ over $k = \R$, the real numbers. 

To compute $\Gt_0(X)$ we consider a compactification
$X \subset \ol{X}$, where $\ol{X}$ is given
by equation $X^2 + Y^2 = 0$ in the real projective
plane $\P^2_{\R}$ with coordinates $X, Y, Z$.
The complement $\ol{X} \setminus X$ is the single closed
(non-rational) point at infinity $\infty \in \ol{X}$;
as a subscheme $\infty$ is isomorphic to $\Spec(\C)$.

It is easy to see that there is an isomorphism
\[
\Gt_0(\ol{X}) \simeq \Z \oplus \Z, \quad [\FF] \mapsto
(\rk(\FF), \deg(\FF))
\]
(surjectivity is obvious, while injectivity
boils down to the fact that every class of a skyscraper sheaf of a closed point $x \in \ol{X}$
is a multiple of the class of the skyscraper sheaf of the rational point $[0:0:1]$,
and this can be checked using the fact that $\CH_0(\ol{X}) = \Z$).

We write the $\Gt$-theory localization sequence for $X \subset \ol{X}$:
\[
\Gt_0(\Spec(\C)) \to \Gt_0(\ol{X}) \to \Gt_0(X) \to 0.
\]
Under the isomorphism $\Gt_0(\ol{X}) = \Z \oplus \Z$
the class of structure sheaf of the point at infinity corresponds to
$(0,2)$; we deduce that $\Gt_0(X) = \Z \oplus \Z_2$
given by the rank map and degree modulo two.

The curve $X$ is $\A^1$-contractible by Lemma \ref{lem:contractible} (2) so that by Proposition
\ref{prop:A1-Ksg} (2) we obtain
\[
\Ksg_0(X) = \Gt_0(X) / \Z \cdot [\OO_X] = \Z_2,
\]
generated by the class of the structure sheaf
of the rational point $(0,0) \in X$.
\end{example}

\subsection{Topological filtration on $\Ksg_0$}

We introduce and study the topological filtration on $\Ksg_0(X)$, that is the filtration given by the codimension of support.
Note that unlike the case of objects from $\Db(X)$, the support of an object in $\Dsg(X)$ is NOT well-defined, as the following example demonstrates.

\begin{example}\label{ex:no-support}
Let $X=\lbrace xy=0\rbrace\subset\A^2$ be the union of two $\A^1$-lines over $k$. Denote the two affine lines by $L_1 = \A^1 \times 0 \subset X$ and $L_2 =0 \times \A^1 \subset X$. 
The structure sheaves of $L_1$ and $L_2$ correspond to quotient rings $k[x] = k[x,y]/(y)$ and $k[y] = k[x,y]/(x)$ respectively.
We have an exact sequence of $k[x,y]/(xy)$-modules
\[
0 \to k[x] \overset{x}{\to} k[x,y]/(x,y) \to k[y] \to 0
\]
which translates into a distinguished triangle in $\Db(X)$
\[
\OO_{L_1} \to \OO_X \to \OO_{L_2} \to  \OO_{L_1}[1]
\]
and yields an isomorphism $\OO_{L_2} \simeq \OO_{L_1}[1]$ of objects $\Dsg(X)$ (the shift [1] is two-periodic in this example); sheaf-theoretic supports of these two objects are different.
\end{example}

We can speak about codimension of support of an object of $\Dsg(X)$ without defining the support itself.
Let $X$ be a quasi-projective scheme with all irreducible
components of the same dimension $n$.

Recall that $\Kt_0(X)$, $\Gt_0(X)$ admit the so-called topological filtration (also called
the coniveau or the codimension filtration), which goes back to
Grothendieck and is defined as follows. 
The class $\alpha \in \Gt_0(X)$ (resp. $\Kt_0(X)$) belongs to $F^i \Gt_0(X)$ (resp. $F^i \Kt_0(X)$)
if $\alpha$ can be represented by a 
bounded complex of coherent sheaves (resp. locally free sheaves) whose support has codimension
at least $i$.
It is clear from the definitions that the canonical map $\can: \Kt_0(X) \to \Gt_0(X)$
maps $F^i \Kt_0(X)$ to $F^i \Gt_0(X)$.

We consider the natural quotient homomorphism $Q:\Gt_0(X)\twoheadrightarrow \Ksg_0(X)$
and define
\[
F^i \Ksg_0(X) = Q(F^i \Gt_0(X)).
\]
This gives a filtration
$$0=F^{n+1}\Ksg_0(X)\subseteq F^{n}\Ksg_0(X)\subseteq \ldots\subseteq F^1 \Ksg_0(X)\subseteq F^0\Ksg_0(X)=\Ksg_0(X).$$

Explicitly, we say that a class $\alpha \in \Ksg_0(X)$ has codimension at least $i$ or that $\alpha \in F^i \Ksg_0(X)$, if
$\alpha$ can be represented by a complex of coherent sheaves $\EE$ on $X$ whose cohomology sheaves are supported in codimension $i$.

It follows from definitions that we have canonical isomorphisms
\begin{equation}\label{eq:F-formula}
F^i \Ksg_0(X) \simeq \frac{F^i \Gt_0(X)}{F^i \Gt_0(X) \cap \can(\Kt_0(X))}
\end{equation}

We let $\Gr^i \Ksg_0(X) = F^{i} \Ksg_0(X) / F^{i+1} \Ksg_0(X)$, and similarly for $\Kt_0(X)$, $\Gt_0(X)$. 
We have a natural surjection
\begin{equation}\label{eq:GrK0-sing}
\Gr^i \Gt_0(X) \to \Gr^i \Ksg_0(X) 
\end{equation}
and canonical isomorphisms
\begin{equation}\label{eq:gr-formula}
\Gr^i \Ksg_0(X) \simeq \frac{F^i \Gt_0(X)}{F^{i+1} \Gt_0(X) + (F^i \Gt_0(X) \cap \can(\Kt_0(X)))}.
\end{equation}

\begin{proposition}\label{prop:K_0sg-filtr}
Let $X/k$ be a connected reduced quasi-projective scheme
with all irreducible components of the same
dimension $n$.
\begin{itemize}
 \item[(1)] Let $N$ be the number of irreducible components of $X$. Then $\Gr^0 \Ksg_0(X) = \Z^{N-1}$.
 In particular $\Gr^0 \Ksg_0(X) = 0$ if and only if $X$ is irreducible.
 \item[(2)] If $X$ is irreducible and normal then $\Gr^1 \Ksg_0(X) \simeq \Cl(X) / \Pic(X)$. 
 In particular in this case $\Gr^1 \Ksg_0(X) = 0$ if and only if $X$ is factorial.
 \item[(3)] For any $i \ge 0$ there is a natural surjective homomorphism $\CH_{n-i}(X) \to \Gr^i \Ksg_0(X)$.
 \item[(4)] If $k$ is algebraically closed, then $\Gr^n \Ksg_0(X) = 0$.
\end{itemize}

\end{proposition}

\begin{proof}
(1) Using $i = 0$ case of \eqref{eq:gr-formula} we obtain
\[
\Gr^0 \Ksg_0(X) \simeq \frac{\Gt_0(X)}{F^1 \Gt_0(X) + \can(\Kt_0(X))}.
\]

We have $\Gr^0 \Gt_0(X) = \Z^N$, where the isomorphism is given by generic rank at the irreducible components.
Since $X$ is connected,
a locally-free sheaf has the same rank at each point, 
and the image of the composition $\Kt_0(X) \overset{\can}{\to} \Gt_0(X) \to \Gr^0 \Gt_0(X)$ consists of $\Z$ embedded into $\Z^N$ diagonally, since $X$ is reduced.
We conclude that $\Gr^0 \Ksg_0(X) = \Z^{N-1}$.

(2) Since $X$ is irreducible, we have canonical splittings $\Kt_0(X) = \Z \oplus F^1 \Kt_0(X)$ and $\Gt_0(X) = \Z \oplus F^1 \Gt_0(X)$, which are respected by $\can$ so 
that from \eqref{eq:gr-formula} we deduce
\[
\Gr^1 \Ksg_0(X) \simeq \frac{F^1 \Gt_0(X)}{F^2 \Gt_0(X) + (F^1 \Gt_0(X) \cap \can(\Kt_0(X)))} 
\simeq \frac{F^1 \Gt_0(X)}{F^2 \Gt_0(X) + \can(F^1 \Kt_0(X))}.
\]

By \cite[Remark 1 on Page 126]{fulton-lang}, we have a natural isomorphism $\Gr^1 \Kt_0(X) = \Pic(X)$.
Since $X$ is normal so that its singular locus has codimension at least two we also get the following isomorphisms
\[
\Gr^1 \Gt_0(X) = \Gr^1 \Gt_0(X \setminus \Sing(X)) = \Pic(X \setminus \Sing(X)) = \Cl(X).                                                                                                                   
\]

It follows that the image of $F^1 \Kt_0(X)$ in $\Gr^1 \Gt_0(X) = \Cl(X)$ is equal 
to $\Pic(X)$, and we get $\Gr^1 \Ksg_0(X) = \Cl(X) / \Pic(X)$.

Finally, $X$ is factorial if and only if $\Pic(X) = \Cl(X)$ which is equivalent to $\Gr^1 \Ksg_0(X) = 0$.

(3) There is a surjection $\CH_{n-i}(X) \to \Gr^i\Gt_0(X)$, 
sending the class of an $(n-i)$-dimensional subvariety to the structure sheaf of this variety 
(see SGA X \cite{sga-6} or \cite[Lemma 3.8, Theorem 3.9]{gillet}),
and composing it with the surjection $\Gr^i \Gt_0(X) \to \Gr^i \Ksg_0(X)$ gives the desired homomorphism.

(4) This is a simple Moving Lemma argument.
Assume first that $X$ is irreducible. 
We fix a closed subvariety $Z \subsetneq X$ containing the singular locus.
By De Jong's work (see \cite{de-jong}, Theorem 4.1)
there is a proper surjective and generically finite morphism $\pi:Y\to X$ where $Y$ is a smooth irreducible and quasi-projective variety.
Let $E = \pi^{-1}(Z) \subset Y$.

Let us show that for every closed point $x \in X$ there is a point $x' \in X \setminus Z$ such that $[\OO_x] = [\OO_{x'}] \in \Gt_0(X)$.
Indeed, since we assume that $k$ is algebraically closed, there is a closed point $y \in Y$ such that $\pi(y) = x$, and using a simple argument (e.g. reducing to the case when $Y$ is a curve,
or using the Moving Lemma \cite{moving-lemma} for Chow groups),
we can find $y' \in Y \setminus E$ such that $[\OO_y] = [\OO_{y'}] \in \Gt_0(Y)$. Pushing forward this equality to $X$ we get $[\OO_x] = [\OO_{x'}] \in \Gt_0(X)$, where $x' = \pi(y')$.
Since the structure sheaves of non-singular points are perfect complexes, we get $[\OO_x] = [\OO_{x'}] = 0 \in \Ksg_0(X)$.
Finally, every class of a zero-dimensional complex $[\FF] \in \Gt_0(X)$ is a linear
combination of structure sheaves of closed points, and this shows that $F^n \Ksg_0(X) = 0$ if $X$ is irreducible.

If $X$ is not irreducible, the result is obtained by applying the argument above to each of the irreducible components of $X_i \subset X$ with respect to $Z = X_i \cap \Sing(X)$.
\end{proof}


\begin{corollary}\label{cor:structure-sheaves-vanish} If $k$ is algebraically closed, $x\in X$ is a 
closed point and $\OO_x$ its structure sheaf, 
then the image of $[\OO_x]\in \Gt_0(X)$ in $\Ksg_0(X)$ is zero. 
\end{corollary}
\begin{proof} 
This is equivalent to Proposition $\ref{prop:K_0sg-filtr}$ (4).
\end{proof}

\begin{remark}
The result of the Corollary does not hold if
$k$ is not algebraically closed: see Example \ref{ex:x2-y2},
where $X/\R$ is a curve, $\Ksg_0(X) = \Z_2$ and the generator is supported in codimension one.
\end{remark}

\begin{corollary}\label{cor:k-0-supp-sing-zero}
If $k$ is an algebraically closed field, and $X$ has isolated singularities then 
$\Ksg_0(X \text{ on } \Sing(X)) = 0$.
\end{corollary}
\begin{proof}
By definition $\Ksg_0(X \text{ on } \Sing(X)) = \Kt_0(\Dsg_{\Sing(X)}(X))$, and from
Proposition \ref{prop:orl-supp} and Theorem \ref{thm:thomason} it follows that
\[
\Ksg_0(X \text{ on } \Sing(X)) \subset \Ksg_0(X) 
\]
is generated by classes of coherent sheaves supported on the singular locus,
so the first group has to be zero by Proposition $\ref{prop:K_0sg-filtr}$ (4)
as the singular locus is zero-dimensional by assumption.
\end{proof}

\begin{corollary}\label{cor:K_0^Sg-of-curves-and-surfaces} Let $k$ be an algebraically closed field.
\begin{itemize}
\item[(1)] If $X$ is a connected quasi-projective curve with $N$ irreducible components, then
$\Ksg_0(X)\simeq\Z^{N-1}.$ In particular, $\Ksg_0(X) = 0$ if and only if $X$ is irreducible.
\item[(2)] If $X$ is a normal irreducible quasi-projective surface, then
$\Ksg_0(X)\simeq\Cl(X)/\Pic(X).$
In particular, $\Ksg_0(X) = 0$ if and only if $X$ is factorial.
\end{itemize}
\end{corollary}

\begin{proof}
We start by noticing that if $\Gr^i \Ksg_0(X)$ is the only nontrivial quotient of the topological filtration, then $\Ksg_0(X) = \Gr^i \Ksg_0(X)$.
(1) follows as $\Gr^i \Ksg_0(X)$ are all zero except for $\Gr^0 \Ksg_0(X) = \Z^{N-1}$, and similarly
(2) follows using irreducibility of $X$ since $\Gr^i \Ksg_0(X)$ are all zero except for $\Gr^1 \Ksg_0(X) = \Cl(X) / \Pic(X)$.
\end{proof}

Recall functoriality of the singularity $\Kt$-groups stated in Lemmas \ref{lem:funct}, \ref{lem:open-iso},
\ref{lem:homotopy-inv}. We now explain how the topological filtration is affected by pull-back and push-forward.

\begin{lemma}\label{lemma:finite-tor-morph}
Let $\phi: X \to Y$ be a morphism of finite $\Tor$-dimension.
\begin{enumerate}
 \item If $\phi$ is flat or a regular closed embedding, then $\phi^*(F^i \Ksg_0(Y)) \subset F^{i} \Ksg_0(X)$.
 \item If $\phi$ is a vector bundle or an open embedding containing the singular locus of $Y$, 
 then $\phi^*: F^i \Ksg_0(Y) \simeq F^i \Ksg_0(X)$.
 \item If $\phi$ is proper of codimension $c:=\dim(Y)-\dim(X)$, then $\phi_*(F^i \Ksg_0(X)) \subset
 F^{i+c} \Ksg_0(Y)$.
\end{enumerate}
\end{lemma}

\begin{proof} 
(1) The result in the case of flat morphisms follows from \cite[Lemma 5.28]{gillet},
while in the case of regular embeddings it follows from \cite[Theorem 5.27]{gillet},

(2) The vector bundle case is \cite[Lemma 5.29]{gillet}. Let $\phi$ be an open embedding,
since it is flat by (1) we have $\phi^*(F^i \Ksg_0(Y)) \subset F^{i} \Ksg_0(X)$
and we need to show that this is an equality.
For that it suffices to show that every coherent sheaf $\FF$ on $X$ with support in codimension
$i$ can be extended to a coherent sheaf $\FF'$ on $Y$ with the same bound on support.

One constructs $\FF '$ as a coherent subsheaf of the quasi-coherent 
sheaf $\phi_*( \FF)$ (see \cite[Ex. Chapter 2, 5.15]{hartshorne}.
Since $\phi_*(\FF)$ is supported on the closure of the support of $\FF$, 
we see that $\FF'$ is supported in codimension $i$.

(3) If $\phi:X\to Y$ is a proper morphism of codimension $c$, then 
$\Supp(\phi_*\EE)\subset\phi(\Supp(\EE))$ and thus $\phi_*:F^i\Gt_0(X)\to F^{i+c}\Gt_0(Y)$ 
(see also \cite{fulton-lang}, Ch. VI, Prop. 5.6)
which implies the result.
\end{proof}

We now explain how Kn\"orrer periodicity (Theorem \ref{Thm:Knorrer}) shifts the topological filtration.

\begin{proposition}\label{prop:knorrer-shift}
The isomorphism $\Ksg_0(Z_f) \simeq \Ksg_0(Z_g)$ induced by
Theorem \ref{Thm:Knorrer} shifts the topological filtration by one, 
that is for all $i \ge 0$ we have natural isomorphisms 
$F^i \Ksg_0(Z_f) \simeq F^{i+1} \Ksg_0(Z_g)$ and $\Gr^i \Ksg_0(Z_f) \simeq \Gr^{i+1} \Ksg_0(Z_g)$.
\end{proposition}

\begin{proof}
We know by Lemma \ref{lemma:finite-tor-morph} that $p^*$ preserves the topological filtration
and that $i_*$ shifts it by one, however this only implies that 
$i_* p^* (F^i \Ksg_0(Z_f)) \subset F^{i+1} \Ksg_0(Z_g)$. To show the equality we give a different
presentation of the Kn\"orrer periodicity isomorphism.

Let $Y = Bl_{Z_f \times 0}(X \times \A^1)$ be the blow up and let $E$ be the exceptional divisor. 
Since $Z_f \times \A^1$
is a complete intersection in $X \times \A^1$ the blow up enjoys the same properties which hold in the smooth case.
For instance, $E$ is a projective bundle $\pi: E \to Z$, and there is a semiorthogonal decomposition
\cite{orlov-monoidal}, \cite[Theorem 6.9]{bergh-schn}
$$\Db(Y)=\langle \Db(Z_f),\Db(X\times\A^1)\rangle.$$

The inclusion of $\Db(Z_f)$ into $\Db(Y)$ 
is given by the fully faithful functor $\Phi: \Db(Z_f)\to \Db(Y)$ 
\[
\Phi(-)= {i_E}_*(\OO_E(-1)\otimes \pi^*(-)),
\]
and its left adjoint is
\[
\Psi(-) = \pi_*(\OO_E(-1)\otimes {i_E}^*(-))[1].
\]

Writing the open charts for the blow up one sees that one of the open charts is isomorphic to $Z_g$ while
the other one is non-singular. We write $j: Z_g \to Y$ for the open embedding of the first open chart;
on the level of singularity categories,
$j^*$ is an equivalence by Proposition $\ref{prop:reduction-open-containing-sing}$. 

We consider the composition $j^* \Phi: \Db(Z_f) \to \Db(Z_g)$ and we now will show that 
$j^* \Phi = i_* p^*$.
We note that restriction of $\pi$ to
$E \cap Z_g$ is a trivial bundle so that $E\cap Z_g = Z_f\times \A^1$ 
(in fact $\pi$ itself is a trivial bundle since the normal bundle of $Z_f \times \{0\}$ in $X \times \A^1$
is trivial).

We consider the cartesian diagram
\begin{equation*}
\xymatrix{
Z_f & \\
E\cap Z_g \ar[u]^p \ar@{->}[d]_{i}\ar@{->}[r]^{j_{E\cap Z_g}} & E\ar@{->}[d]^{i_E} \ar[ul]_\pi \\
Z_g\ar@{->}[r]^j & Y}
\end{equation*} 
and by flat base change 
we compute that
$$j^*\Phi(-)=j^* {i_E}_*(\OO_E(-1)\otimes \pi^*(-))=
 i_*( j_{E\cap Z_g}^*\OO_E(-1)\otimes j_{E\cap Z_g}^*\pi^*(-))= i_*p^*(-).$$
where we used that $j_{E\cap Z_g}^*\OO_E(-1)\simeq \OO_{E\cap Z_g}$.

Now we check the effect of $j^* \Phi$ on the toplogical filtration. We rely on Lemma \ref{lemma:finite-tor-morph}.
Since $j^*$ strictly preserves
the filtration it is sufficient to check that $\Phi$ strictly shifts the filtration by one: this holds true since
$\pi^*$ preserves the filtration while ${i_{E}}_*$ shifts it by one, and
the left adjoint $\Psi$ of $\Phi$ which will become its inverse on the level of singularity categories,
shifts the filtration by negative one: this
holds since $i_E^*$ preserves the filtration while $\pi_*$ shifts it by negative one.
\end{proof}

The next two examples consider the singularity Grothendieck group of
split nodal affine quadrics, that is ordinary double points (cf \cite[3.3]{orlov-sing-1}).

\begin{example}[Even-dimensional ordinary double points]\label{ex:A-ev}
Let $n = 2m$ and consider
the split quadratic form
\begin{equation*}
q_n = 
\sum x_iy_i+z^2\in k[x_1,y_1,\ldots ,x_m,y_m,z] 
\end{equation*}
and let $Q_n\subset\A^{n+1}$ be the $n$-dimensional nodal quadric defined by $q_n = 0$.

From Kn\"orrer periodicity we get 
\[
\Ksg_0(Q_n) \simeq \Ksg_0(k[z]/(z^2)) = \Z_2,
\]
see Example \ref{ex:k-eps}.
Furthermore, since by dimension reasons the nonzero element
of $\Ksg_0(k[z]/(z^2))$ has support in codimension zero, by the shift of the topological filtration 
of Proposition \ref{prop:knorrer-shift} we get
\[
\Ksg_0(Q_n) = \Gr^{n/2} \Ksg_0(Q_n) = \Z_2.
\]
Explicitly $\Ksg_0(Q_n)$ can be seen to be generated by the structure sheaf of $n/2$-codimensional subvariety $V(y_1,\ldots , y_m, z) \subset Q_n$.
\end{example}

\begin{example}[Odd-dimensional ordinary double points] \label{ex:A-odd}
Let $n = 2m-1$ and consider
\begin{equation*}
q_n = \sum x_iy_i\in k[x_1,y_1,\ldots ,x_m ,y_m]  
\end{equation*}
and let $Q_n\subset\A^{n+1}$ be the $n$-dimensional nodal quadric defined by $q_n = 0$.

By Kn\"orrer periodicity we get 
\begin{equation*}
\Ksg_0(Q_n) \simeq \Ksg_0(k[x,y]/(xy)),
\end{equation*}
and the latter Grothendieck group is isomorphic to $\Z$ by Corollary
\ref{cor:K_0^Sg-of-curves-and-surfaces} (1), 
generated by the structure sheaves of one of the two irreducible components 
(cf Example \ref{ex:no-support}).
Using Proposition \ref{prop:knorrer-shift} we obtain
\[
\Ksg_0(Q_n) = \Gr^{(n-1)/2} \Ksg_0(Q_n) = \Z,
\]
generated by the structure sheaf of a codimension $(n-1)/2$ linear space $V(y_1, \dots, y_m) \subset Q_n$.
\end{example}

\section{Singularity $\Kt$-theory of quotient singularities}

\subsection{The local case: non-positive $\Kt$-groups}

For a finite group $G \subset \GL_n(k)$ we consider the quotient variety
$\A^n/G=\Spec(k[x_1,\ldots ,x_n]^G)$.
In this subsection we study $\Gt_0(\A^n/G)$
as well as $\Kt$-groups $\Kt_i(\A^n/G)$ and $\Ksg_i(\A^n/G)$ for $i\leq 0$.

\begin{proposition}\label{prop:K_-1} Assume that $\A^n/G$ has an isolated singularity at the origin. Then
$$\Kt_0(\mathbb{A}^n/G)\simeq\mathbb{Z}\quad and\quad \Kt_{-j}(\mathbb{A}^n/G)=\Ksg_{-j}(\mathbb{A}^n/G)=0\ \text{for}\ j > 0.$$
\end{proposition}

\begin{proof} 


Since $k[x_1,\ldots ,x_n]^G$ is a positively graded algebra, we can use 
\cite[Theorem 1.2]{weibel-co} to express non-positive $K$-theory  groups
as follows:
$$\Kt_0(\A^n/G)=\Z\oplus\Pic(\A^n/G)\oplus\bigoplus_{i=1}^{n-1} H^i_{\cdh}(\A^n/G,\Omega^i_{/ \Q})/dH^i_{\cdh}(\A^n/G,\Omega^{i-1}_{/ \Q})$$
and for $j > 0$
$$\Kt_{-j}(\A^n/G)= H_{\cdh}^j(\A^n/G,\OO)\oplus\bigoplus_{i=1}^{n-j-1} H^{i+j}_{\cdh}(\A^n/G,\Omega^i_{/ \Q})/dH^{i+j}_{\cdh}(\A^n/G,\Omega^{i-1}_{/ \Q}).$$
Here $H^*_{\cdh}$ denotes cohomology of $\A^n/G$ defined via the $\cdh$-topology on $\Sch /k$ \cite{suslin-voevodsky} and the group $dH^j_{\cdh}(\A^n/G,\Omega^{i-1}_{/ \Q})$ is the image of the map $d:H^j_{\cdh}(\A^n/G,\Omega^{i-1}_{/ \Q})\to H^j_{\cdh}(\A^n/G,\Omega^{i}_{/ \Q})$ induced by the K\"ahler differential.

Let us show that cohomology groups $H^q(\A^n/G, \Omega^{p}_{/ \Q})$ are zero for
all $q > 0$, $p \ge 0$.
Since every $G$-representation is defined over
a subfield $k_0 \subset k$ which is a finite extension of $\Q$, 
$\A^n / G$ admits a model over $k_0$.

Let us show that for every variety $X/k$ admitting a model $X_0/k_0$
we have 
the K\"unneth formula for $\cdh$ cohomology groups (cf. \cite[Corollary 4.5]{weibel-co-toric}):
\begin{equation}\label{eq:cdh-basechange}
H^q_{\cdh}(X,\Omega^p_{/ \Q})\simeq \bigoplus_{p=i+j}H^q_{\cdh}(X,\Omega^i_{/k})\otimes_k \Omega^j_{k/ \Q},
\end{equation}
for $p,q\geq 0$.

Let us first assume that $X$ is smooth.
In this case using \cite[Corollary 2.5]{weibel-co-vorst}, 
\eqref{eq:cdh-basechange} is equivalent to the isomorphism of Zariski
cohomology groups
\begin{equation}\label{eq:cdh-basechange-Zar}
H^q(X,\Omega^p_{X/ \Q})\simeq \bigoplus_{p=i+j}H^q(X,\Omega^i_{X/k})\otimes_k \Omega^j_{k/ \Q}.
\end{equation}

Note that as $k_0$ is finite separable over $\Q$ we may replace
$\Q$ by $k_0$ in \eqref{eq:cdh-basechange-Zar}.
Since $X$ has a model over $k_0$ we have a splitting
$\Omega^1_{X/ k_0}\simeq \Omega^1_{X/k}\oplus (\Omega^1_{k/ k_0} \otimes \OO_X)$ 
and hence
$$\Omega^p_{X/ k_0}\simeq \bigoplus_{p=i+j} \Omega^i_{X/k}\otimes \Omega^j_{k/ k_0}.$$
Applying cohomology one obtains 
\eqref{eq:cdh-basechange-Zar} and hence \eqref{eq:cdh-basechange} in the smooth case.

In general, that is when $X$ is singular, \eqref{eq:cdh-basechange} follows by induction on dimension using the cdh descent exact sequence \cite[12.1]{suslin-voevodsky} applied to a resolution of $X$
definable over $k_0$.

Finally, by \cite[Lemma 2.1]{weibel-co-negative-bass} and \cite[Theorem 5.12]{du-bois} we see that
$$H_{\cdh}^q(\A^n/G ,\Omega^p_{/k})\simeq  H^q(\A^n,\Omega^p_{\A^n/k})^G,$$
and since $\A^n$ is affine, these cohomology groups vanish for $q > 0$.
So combining everything together we get $\Kt_{-j}(\A^n/G)=0$ and $\Kt_0(\A^n/G)=\Z\oplus\Pic(\A^n/G) = \Z$,
since the Picard group of a normal graded $k$-algebra is zero.
\end{proof}

\begin{remark}
In the previous version of this paper we claimed that every vector bundle on
$\A^n/G$ is trivial. We do not know if this statement is true. We thank
Sasha Kuznetsov for pointing out an error in our argument.
\end{remark}

\begin{corollary}\label{cor:A^n/G-idempotent-complete} 
If $\A^n/G$ is an isolated singularity, then the singularity category $\Dsg(\A^n/G)$ is idempotent complete.
\end{corollary}
\begin{proof}
As $\Kt_{-1}(\A^n/G) = 0$ by Proposition \ref{prop:K_-1}, the result follows from Lemma \ref{lemma:idempotent-K-1}.
\end{proof}

\begin{remark}
It is not true that every affine quasi-homogeneous or $\A^1$-contractible singularity has an idempotent complete
singularity category: the simplest example is provided
by the so-called Bloch-Murthy surface singularity $X$ given by 
$x^2 + y^3 + z^7 = 0$
which has non-vanishing $K_{-1}(X)$ \cite[Example 6.1]{weibel-surfaces}.
\end{remark}


\begin{proposition}\label{prop:G_0} Let $G$ be a finite group acting linearly on the affine space $\mathbb{A}^n$ over a field $k$. Then we have
\[
\Gt_0(\A^n / G) = \Z \oplus \Ksg_0(\A^n / G),
\] 
and $\Ksg_0(\A^n / G)$ is a finite torsion group.
\end{proposition} 

\begin{proof}
By Proposition \ref{prop:A1-Ksg} (3) there is a split short exact sequence
\[
0 \to \Z \to \Gt_0(\A^n / G) \to \Ksg_0(\A^n / G) \to 0 
\]
where the first map is split by the rank map. 
Let us show that $\Ksg_0(\A^n / G)$ is finite torsion. 
The fact that $\Gt_0(\A^n / G)$ is finitely-generated is well-known \cite{auslander-reiten} and 
follows e.g. from the fact that the push-forward
functor from the equivariant category to the category of coherent sheaves on the quotient variety 
$\pi_*: \Db_G(\A^n) \to \Db(\A^n/G)$ is essentially surjective
by Theorem \ref{thm:res-of-sing-ess-surjective} (1).
Thus it suffices to show that $\Gt_0(\A^n / G)$ is of rank one.
For that we can work rationally and compare $\Gt_0$ to the Chow groups. We have a chain of isomorphisms
\[
\Gt_0(\A^n / G) \otimes \Q \simeq \CH_*(\A^n/G) \otimes \Q \simeq \CH_*(\A^n)^G \otimes \Q \simeq \Q,
\]
where the first isomorphism is the Grothendieck-Riemann-Roch Theorem for singular varieties \cite[Chapter III]{baum-fulton-macpherson}
and the second isomorphism is \cite[Example 1.7.6]{fulton}. We conclude that $\Gt_0(\A^n / G)$ is a finitely generated abelian group of rank one
and that $\Ksg_0(\A^n / G)$ is finite torsion.
\end{proof}

We call an element $g \in \GL_n(k)$
a reflection if $g$ has finite order and acts trivially on a hyperplane.
We need the following well-known Lemmas.

\begin{lemma}[{\cite[Theorem 3.9.2]{benson}}]\label{lemma:Cl-group}
Let $G$ be a finite subgroup of $\GL_n(k)$ 
and let $N$ be the subgroup of $G$ generated by reflections.
There is a natural isomorphism $\Cl(\A^n / G) \simeq \wh{G / N}$.
\end{lemma}
\begin{proof}
Our proof relies on equivariant Chow groups \cite{edidingraham}. Let us first assume that $G$ does not contain reflections.
In this case there is a $G$-invariant subvariety $Z\subset\mathbb{A}^n$ of codimension at least two, such that $G$ acts freely on $\mathbb{A}^n \setminus Z$. 
Let $\pi: \A^n \to \A^n / G$ be the quotient map.
Since removing locus of codimension two does not change $(n-1)$-st Chow groups,
we have a chain of isomorphisms
\[
\Cl(\A^n / G) = \CH_{n-1}(\A^n / G) \simeq \CH_{n-1}(\A^n /G \setminus \pi(Z)) = \CH_{n-1}((\A^n \setminus Z) / G) \simeq \CH_{n-1}^G(\A^n).
\]
Since $\mathbb{A}^n$ is smooth, we have $\CH_{n-1}^G(\mathbb{A}^n) = \Pic^G(\mathbb{A}^n)$,
and the latter group of $G$-equivariant line bundles on $\mathbb{A}^n$ is isomorphic to the group $\wh{G}$ of characters of $G$.

In the general case the fixed locus of the action of $G/N$ on $\A^n / N \simeq \A^n$ does not contain
divisors and the same argument applies to show that $\Cl(\A^n / G) \simeq \wh{G/N}$.
\end{proof}

\begin{lemma}\label{lemma:CH-G-torsion}
For every $0 \le i \le n-1$, $\CH_i(\A^n / G)$ is annihilated by $|G|$. 
\end{lemma}
\begin{proof}
Let $V \subset \A^n / G$ be a subvariety, let $\pi^{-1}(V)$ be the scheme theoretic preimage of $V$
under the quotient morphism $\pi: \A^n \to \A^n/G$, and let $W$ be a reduced irreducible component of $\pi^{-1}(V)$.
Let us show that the degree of the field extension $[k(W):k(V)]$ divides $|G|$.

Since we assume $k$ to be of characteristic zero, $G$
is a linearly reductive $k$-group scheme.
Thus according to \cite[Proof of Theorem 1.1, p. 28]{Mumford}, the $G$-invariant ring of 
$A := k[x_1,\ldots ,x_n]\otimes_{k[\A^n/G]} k(V)$ is just $k(V)$.
If $A_{red}$ is the quotient of $A$ by the nilradical, then 
\[
A_{red} = \prod_{i=1}^r k(W_i)
\]
where $W_1, \dots, W_r$ are all components of $\pi^{-1}(V)$. The action of $G$ on $A$ induces an action on $A_{red}$, 
and $A_{red}^G = A^G = k(V)$.
Since $G$ acts on the components $W_1, \dots, W_r$ transitively, 
the degrees $[k(W_i):k(V)]$ are equal to each other, and it is easy to see that they
divide $|G|$.

Since $\CH_i(\A^n) = 0$ in the considered range,
by definition of push-forward on Chow groups we get
\[
0 = \pi_*([W]) = [k(W):k(V)] \cdot [V] 
\]
so that $[V] \in \CH_i(\A^n/G)$ is $|G|$-torsion.
\end{proof}

\begin{proposition}\label{prop:AnG-c1}
There is a well-defined surjective first Chern class homomorphism
\[
c_1: \Ksg_0(\A^n / G) \to \wh{G/N}. 
\]
If $n = 2$ and $k$ is algebraically closed, then $c_1$ is an isomorphism.
\end{proposition}
\begin{proof}
We claim that there is a Grothendieck-Riemann-Roch without denominators style surjection
$$(\rk,c_1):\Gt_0(\mathbb{A}^n/G)\twoheadrightarrow\mathbb{Z}\oplus\CH_{n-1}(\mathbb{A}^n/G).$$
Indeed, since $\mathbb{A}^n/G$ is normal, 
to construct $c_1$ one may simply remove the singular locus of $\A^n / G$ and thus reduce to the smooth case,
and the surjectivity follows easily.

Splitting off the direct summand $\Z$ corresponding to the trivial bundles and using Proposition \ref{prop:G_0} and Lemma \ref{lemma:Cl-group}
we get the desired surjection. 

By construction of the topological filtration on $\Ksg_0(X)$ we have $\Ker(c_1) = F^2 \Ksg_0(X)$ (cf. proof of Proposition \ref{prop:K_0sg-filtr} (2)),
in particular if $n=2$, then $\Ker(c_1) = F^2 \Ksg_0(X) = 0$ by Proposition \ref{prop:K_0sg-filtr} (4).
\end{proof}

\begin{proposition}\label{prop:K_0^Sg-and-order}
Let $k$ be an algebraically closed field of characteristic zero. Then every element of $\Ksg_0(\A^n / G)$ is annihilated by $|G|^{n-1}$.
\end{proposition}
\begin{proof}
We consider the topological filtration on $\Ksg_0(\A^n/G)$ and its associated graded pieces $\Gr^i \Ksg_0(\mathbb{A}^n/G)$.
By Proposition \ref{prop:K_0sg-filtr} we have $\Gr^0  \Ksg_0(\mathbb{A}^n/G) = \Gr^n \Ksg_0(\mathbb{A}^n/G) = 0$ so that the filtration has the form
\[
0 = F^n \Ksg_0(\mathbb{A}^n/G)\subset F^{n-1} \Ksg_0(\mathbb{A}^n/G)\subset\ldots \subset F^1 \Ksg_0(\mathbb{A}^n/G) = \Ksg_0(\A^n / G).
\]

By Proposition \ref{prop:K_0sg-filtr}, each subquotient $\Gr^i \Ksg_0(\A^n/G)$, $1 \le i \le n-1$, admits a surjection $\CH_{n-i}(\A^n/G) \to \Gr^i \Ksg_0(\A^n/G)$
and by Lemma \ref{lemma:CH-G-torsion}, each of these groups is annihilated by $|G|$.
This means that multiplication by $|G|$ shifts the filtration: $|G| \cdot F^i \Gt_0(\mathbb{A}^n/G) \subset F^{i+1} \Gt_0(\mathbb{A}^n/G)$, in particular multiplication by $|G|^{n-1}$
acts as the zero map.
\end{proof}





The next proposition gives the formula for $\Ksg_0(\A^n / G)$ in the isolated singularity case.
For other approaches to how to compute this group see \cite{auslander-reiten}, 
\cite{marcos}, \cite{herzog-marcos-waldi}.

\begin{proposition}[\cite{gsprverdier}]\label{prop:gsprverdier} 
Let $G$ be a finite group acting linearly on $\mathbb{A}^n$ such that the 
$G$-action on $\A^n \setminus \{ 0 \}$ is free.
Let $\rho$ be the corresponding representation of $G$.
Then we have
$$\Gt_0(\mathbb{A}^n/G)\simeq R(G)/rR(G),$$
where $R(G)$ is the representation ring of $G$ and $r \in R(G)$ is the Koszul class
\[
r = \sum_{i = 0}^n (-1)^i [\Lambda^i (\rho^{\vee})].
\]
\end{proposition}

\begin{proof} The proof uses equivariant algebraic $\Kt$-theory \cite{thomason-equivariant-I}. 
Since $G$ acts freely away from $0$, there is an isomorphism
$$\Kt_0^G(\mathbb{A}^n \setminus \{0\})\simeq \Gt_0\left(\frac{\mathbb{A}^n \setminus \{0\}}{G}\right).$$
Let $i: \{0\} \to \A^n$, $\ol{i}: \{0\} \to \A^n/G$  be the closed embeddings.
Consider the localization exact sequences of $\Gt_0$ and $\Kt_0^G$:
\begin{equation}\label{eq:K-equiv-seq}
\xymatrix{
\Kt_0^G(0)\ar@{->}[r]^{i_*^G} \ar[d] & \Kt_0^G(\mathbb{A}^n) \ar@{->}[r] \ar[d] & \Kt_0^G(\mathbb{A}^n \setminus \{0\}) \ar@{->}[r] \ar@{->}[d]^{\simeq}  & 0 \\
\Gt_0( 0 ) \ar@{->}[r]^{\ol{i}_*}  & \Gt_0(\mathbb{A}^n/G) \ar@{->}[r] &  \Gt_0(\mathbb{A}^n/G -0)\ar@{->}[r] & 0\\
}\end{equation}
related by push-forward maps followed by taking $G$-invariants.
The push-forward $\ol{i}_*$ is a zero map since it factor through $i_* = 0$.
By equivariant homotopy invariance \cite[4.1]{thomason-equivariant-I} we have an isomorphism $\Kt_0^G(\A^n / G) \simeq \Kt_0^G(0) \simeq R(G)$ and under these identifications
the push-forward $i_*^G: \Kt_0^G(0)\to \Kt_0^G(\mathbb{A}^n)$ corresponds to the multiplication by the class $[\OO_0] = r$.

Putting everything together we obtain
\[
\Gt_0(\mathbb{A}^n/G) \simeq \Gt_0((\A^n \setminus \{0\})/G) \simeq R(G)/rR(G). 
\]
\end{proof}

\begin{remark}\label{rem:isolated-quot}
If $G$ has no reflections, then the condition that $G$ acts freely on $\A^n \setminus \{0\}$
is equivalent to the quotient $\A^n/G$ to have an isolated singularity at the origin.
\end{remark}

\begin{remark}\label{remark:gspverdier} 
%
%
Since $\Ksg_0(\A^n/G)$ is a finite group by
Proposition \ref{prop:G_0}
we see that under the assumptions of Proposition \ref{prop:gsprverdier} the linear map 
$r: R(G) \to R(G)$ has cokernel of rank one, and
so it has a one-dimensional kernel.

\end{remark}

\begin{example}[\cite{gsprverdier}]
Computing $R(G)/rR(G)$ for a two-dimensional $ADE$ singularity $\A^2/G$, one can compute $\Ksg_0(\A^2/G)$ using Proposition \ref{prop:gsprverdier} as follows:

\medskip
\begin{center}   
 
    \begin{tabular}{ | c  | c  |  }
    \hline
    Type & $\Ksg_0(\A^{2}/G)$ \\ \hline
    \hline
     $A_n$ & $\mathbb{Z}/(n+1)\mathbb{Z}$ \\ \hline
     $D_n$, $n$ even & $\mathbb{Z}/2\mathbb{Z}\times\mathbb{Z}/2\mathbb{Z}$  \\ \hline
  $D_n$, $n$ odd &  $\mathbb{Z}/4\mathbb{Z}$  \\ \hline
         $E_6$ & $\mathbb{Z}/3\mathbb{Z}$ \\
    \hline
         $E_7$ & $\mathbb{Z}/2\mathbb{Z}$  \\
    \hline
         $E_8$ & $0$  \\
    \hline
    \end{tabular}
\end{center}
\medskip
(note a typo in \cite{gsprverdier} in the $E_7$ case on page 415). The same result can be obtained using Proposition \ref{prop:AnG-c1},
and another way is given by Yoshino using Auslander-Reiten sequences \cite[Chapter 13]{Yoshino}.
\end{example}


\begin{corollary}
Let $k$ be an algebraically closed field of characteristic zero and let
$X$ be the local $\frac1{m}(\overbrace{1, \dots, 1}^n)$ singularity, that is $X = \A^n / \Z_m$ with the diagonal action by the primitive root of unity. 
Then $\Ksg_0(X)$ is a finite abelian group of order $m^{n-1}$.
\end{corollary}
\begin{proof}
Let us fix a primitive character $\rho$ of $\Z_m$.
Then the representation ring is $R(\Z_m)=\mathbb{Z}[x]/(x^m-1)$ where 
we choose $x$ to be the class $[\rho^\vee]$.
Then $$r=\sum_{i=0}^n (-1)^i[\Lambda^i (\rho^{\vee})]=\sum_{i=0}^n \binom{n}{i}(-1)^ix^i =(1-x)^n\in R(G)$$ 
and after making a substitution $y=1-x$, we obtain
\[
\Gt_0(X) = \Z[y] / (y^n, my - \binom{m}{2}y^2 + \dots) = \Z \cdot 1 \oplus \Ksg_0(X)
\]
so that $\Ksg_0(X)$ is a quotient of a free $\Z$-module with the basis $y, y^2, \dots, y^{n-1}$ by the upper-triangular relations $my^i - \binom{m}{2}y^{i+1} + \dots$ for $i \ge 1$.
This means that $\Ksg_0(X)$ is a finite abelian group of order $m^{n-1}$.
\end{proof}

We abuse the notation slightly by writing $X = \frac1{m}({1, \dots, 1})$ for the corresponding local singularity.
The precise structure of $\Ksg_0(\frac1{m}(\overbrace{1, \dots, 1}^n))$ will vary depending on $m$ and $n$.

\begin{example}
For $n=2$ we have $\Ksg_0(\frac1m(1,1)) \simeq \Z_m$, in accordance with Proposition \ref{prop:K_0^Sg-and-order}.
\end{example}

\begin{example}
For $n=3$ one can see that
\[
\Ksg_0(\tfrac1{m}(1,1,1)) = \left\{\begin{array}{ll}
                       (\Z_m)^2, & \text{$m$ odd} \\
                       \Z_{m/2} \times \Z_{2m}, & \text{$m$ even} \\
                     \end{array}\right. 
\]
\end{example}

\begin{example}
If $m = 2$ and $n$ is arbitrary, one can see that $\Ksg_0(\A^n / \Z_2) \simeq \Z_{2^{n-1}}$
(here the action of $\Z_2$ on $\A^n$ is $v \mapsto -v$).
\end{example}

\subsection{The local case: positive $\Kt$-groups}


\begin{proposition}\label{prop:K_j^Sg} 
Let $k$ be an algebraically closed field of characteristic zero, and
let $G \subset \GL_n(k)$ be a finite group such that the $G$-action on $\A^n \setminus \{ 0 \}$ is free.
For every $j \ge 0$ consider the group $T_j = \Tor(\Ksg_0(\A^n/G), \Kt_j(k))$.

(1) $T_j$ is a finite torsion group annihilated by $|G|^{n-1}$, and
$T_j = 0$ for all even $j$.

(2) For every $j \ge 1$ there is a short exact sequence
\begin{equation}\label{eq:G_j-AnG}
0 \to \Kt_j(k) \to \Gt_j(\A^n/G) \to T_{j-1} \to 0,
\end{equation}
where the first map is the pull-back from $\Spec(k)$.
In particular for all $j \ge 1$ we have 
$\Gt_j(\A^n/G)\otimes \Z[1/|G|] \simeq 
\Kt_j(k) \otimes \Z[1/|G|]$
and for all odd $j \ge 1$ we have
$\Gt_j(\A^n / G) \simeq \Kt_j(k)$.

(3) For every $j \ge 1$, there is an exact sequence
\begin{equation}\label{eq:K_jsg-AnG}
0 \to T_{j-1} \to \Ksg_j(\A^n/G) \to \Kt_{j-1}(\A^n/G) \to \Kt_{j-1}(k) \to 0,
\end{equation}
where the last morphism in the sequence is induced by restriction 
to the rational point $0 \in \A^n/G$.
\end{proposition}

We prove the Proposition
at the end of this subsection.

\begin{corollary}\label{K_1^Sg=0} 
If $\mathbb{A}^n/G$ is an isolated singularity over an algebraically closed field of characteristic zero
then $\Ksg_1(\mathbb{A}^n/G)=0$.
\end{corollary}
\begin{proof} 
We may assume that $G$ acts freely on $\A^n \setminus \{0\}$ (see Remark \ref{rem:isolated-quot}
and proof of Proposition \ref{prop:AnG-c1}).
The result follows from \eqref{eq:K_jsg-AnG} using the fact that 
$T_0 = 0$ and Proposition \ref{prop:K_-1} which says that $\Kt_0(\A^n/G) = \Kt_0(k) = \Z$.
\end{proof}

\begin{remark}
For non-algebraically closed field, see Example \ref{ex:A1-any-field}.
We do not know if $\Ksg_1(\A^n/G) = 0$ in the non-isolated singularity case.
\end{remark}

\begin{remark}\label{rem:K-1-huge}
The structure of the groups $\Kt_j(\A^n/G)$ for $j \ge 1$ is in general not known.
Since the work of Srinivas it is known that $\Ker(\Kt_1(\A^n/G) \to \Kt_1(k))$ is ``huge'', 
that is as large as the base field $k$, even in the simplest $\frac12(1,1)$ case \cite{srinivas3},
and the same follows for $\Kt_2^{sg}(\A^n/G)$ from Proposition \ref{prop:K_j^Sg}.
\end{remark}

In order to prove Proposition \ref{prop:K_j^Sg},
we use the language of equivariant $\Kt$-theory \cite{thomason-equivariant-I}
which for finite groups can also be interpreted as $\Kt$-theory of Deligne-Mumford stacks
\cite{Joshua-Krishna}.

\begin{lemma}\label{lem:r_j}
In the assumptions of Proposition \ref{prop:K_j^Sg}, let $i: \Spec(k) \to \A^n$ be the closed embedding of the 
origin $0$. Then the following is true.

(1) We have natural $R(G)$-module isomorphisms $K_j^G(\A^n) \overset{i^*_G}{\simeq} K_j^G(k) \simeq R(G) \otimes K_j(k)$
and a commutative diagram
\[
\xymatrix{
\Kt_j^G(k) \ar[r]^{i^G_*} & \Kt_j^G(\A^n) \\ 
R(G) \otimes \Kt_j(k) \ar[u]^\simeq \ar[r]^{r_j} & R(G) \otimes \Kt_j(k) \ar[u]^\simeq \\
}
\]
where $r_j$ is multiplication by the Koszul class $r \in R(G)$ defined in Proposition \ref{prop:gsprverdier}.

(2) Let $\pi_{0}$ be the projection from the Deligne-Mumford stack $[\Spec(k)/G]$ to its coarse
moduli space $\Spec(k)$ and let $\alpha_j$ be the restriction of $\pi_{0,*}: \Kt_j^G(k) \to \Kt_j(k)$
to $\Ker(i_*^G) = \Ker(r_j)$. 
Then for every $j \ge 0$ there is a commutative diagram
\begin{equation}\label{eq:alpha}
\xymatrix{
0 \ar[r] & \Kt_j(k) \ar[r] \ar[dr]_{\simeq} & \Ker(r_j) \ar[d]^{\alpha_j} \ar[r]& T_j \ar[r] & 0\\  
& & \Kt_j(k) &  & \\
}
\end{equation}
where $T_j$ is defined as in Proposition \ref{prop:K_j^Sg}
and the top row is exact.

(3) Let $\pi$ be the projection from the Deligne-Mumford stack $[\A^n/G]$ to its coarse moduli space
$\A^n/G$. For every $j \ge 0$ consider the subgroup $\Kt_j(k) \simeq 1 \otimes \Kt_j(k) \subset R(G) \otimes \Kt_j(k) =
\Kt_j^G(\A^n)$,
and let $\beta_j$ be the restriction of $\pi_*: \Kt_j^G(\A^n) \to \Gt_j(\A^n/G)$ 
to this subgroup.
Then $\beta_j$ is isomorphic to pullback morphism $p^*: \Kt_j(k) \to \Gt_j(\A^n/G)$ induced by 
the structure morphism $p: \A^n/G \to \Spec(k)$.

Furthermore, for $j \ge 1$ the embedding $1 \otimes \Kt_j(k) \subset R(G) \otimes K_j(k)$
induces an isomorphism $\Kt_j(k) \simeq \Coker(r_j)$.
\end{lemma}

\begin{proof}
(1) $i^{*}_G$ being an isomorphism is the standard homotopy invariance of $K$-theory in the regular case \cite[4.1]{thomason-equivariant-I},
$\Kt_j^G(k) \simeq R(G) \otimes \Kt_j(k)$ holds e.g. by \cite[Proposition 1.6]{vistoli}.
The commutative diagram follows from \cite[Lemma 1.7]{vistoli}.

(2) Since the map $r_j$ is isomorphic to $r \otimes \id$, $\Ker(r_j)$ can be computed
via the Universal Coefficient Theorem applied to the complex $[r: R(G) \to R(G)]$
as follows. We have
\[
0 \to \Ker(r) \otimes K_j(k) \to \Ker(r_j) \to \Tor(\Coker(r), K_j(k)) \to 0. 
\]
By Remark \ref{remark:gspverdier},
$\Ker(r) = \Z \cdot t$, for some element $t \in R(G)$, and by Proposition \ref{prop:gsprverdier},
$\Coker(r) \simeq \Z \oplus \Ksg_0(\A^n/G)$. We see that
$\Ker(r) \otimes \Kt_j(k) = \Kt_j(k)$ and
\[
\Tor(\Coker(r), K_j(k)) = \Tor(\Ksg_0(\A^n/G)), \Kt_j(k)) = T_j
\]
so that the top row of \eqref{eq:alpha} is exact.
We compute $\alpha_j$ as follows
\[
\alpha_j\vert_{t \otimes \Kt_j(k)} = \pi_{0,*}\vert_{t \otimes \Kt_j(k)} = \pi_{0,*}(t) \cdot \id_{\Kt_j(k)},
\]
and for commutativity of \eqref{eq:alpha} it remains to show that $\pi_{0,*}(t) = \pm 1$. 
This follows easily by extending the commutative diagram \eqref{eq:K-equiv-seq}
on term to the left \cite[Theorem 2.7]{thomason-equivariant-I}
(cf $j=1$ case in \eqref{eq:ladder} in the Proof of Proposition \ref{prop:K_j^Sg}).

(3) The fact that $\beta_j$ is equal to $p^*$ follows from the projection formula.
Since tensor product is right exact we have 
\[
\Coker(r_j) = \Coker(r) \otimes \Kt_j(k) \simeq (\Z \oplus \Ksg_0(\A^n/G)) \otimes \Kt_j(k).
\]
Since $k$ is algebraically closed, by a result of Suslin \cite{suslin}, for every $j \ge 1$, $\Kt_j(k)$
is a divisible group, so that since $\Ksg_0(\A^n/G)$ is torsion, $\Ksg_0(\A^n/G)) \otimes \Kt_j(k) = 0$,
and we have
\[
\Coker(r_j) \simeq \Kt_j(k),
\]
induced by tensoring $\rk: R(G) \to \Z$ by $\Kt_j(k)$. Since $\rk(1) = 1$, the composition 
\[
1 \otimes \Kt_j(k) \subset R(G) \otimes \Kt_j(k) \to \Coker(r_j)
\]
provides a splitting, and hence the inverse to this morphism.
\end{proof}

\begin{proof}[Proof of Proposition \ref{prop:K_j^Sg}]
(1) By Proposition \ref{prop:G_0},
$\Ksg_0(\A^n/G)$ is a torsion group annihilated by
$|G|^{n-1}$, hence the same is true for $T_j$.

For even $j$, $\Kt_j(k)$ of an algebraically closed
field is torsion-free by a result of Suslin \cite{suslin},
hence $T_j$ = 0 for even $j$.

For odd $j$, and every $n \ge 1$, the $n$-torsion subgroup $\Kt_j(k)$ is finite \cite{suslin},
and since $\Ksg_0(\A^n/G)$ is a finite abelian group, $T_j$ is finite as well.

(2) The key in proving \eqref{eq:G_j-AnG} is to compare the localization sequence for $\Gt$-theory of 
$\A^n/G$ to the $G$-equivariant localization sequence for $\Kt$-theory of $\A^n$. The two
sequences are related by pushforward morphisms:
\begin{equation}\label{eq:ladder}
\xymatrix{
\Kt_j^G(k) \ar[r]^{i_*^G} \ar[d]_{\pi_{0,*}} & \Kt_j^G(\A^n) \ar[r] \ar[d]_{\pi_{*}} & \Kt_j^G(\A^n \setminus \{0\}) \ar[r] \ar[d] & \Kt_{j-1}^G(k) \ar[r]^{i_*^G} \ar[d]_{\pi_{0,*}} & \Kt_{j-1}^G(\A^n) \ar[d]_{\pi_{*}} \\
\Kt_j(k) \ar[r]^{i_*} & \Gt_j(\A^n/G) \ar[r] & \Gt_j((\A^n \setminus \{0\})/G) \ar[r] & \Kt_{j-1}(k) \ar[r]^{i_*} & \Gt_{j-1}(\A^n/G) \\
}
\end{equation}

Here $i: \Spec(k) \to \A^n/G$ is the origin (unique fixed point of the action), and $\pi$ (resp. $\pi_0$)
is the canonical morphism from the quotient
Deligne-Mumford stack $[\A^n/G]$ (resp. $[\Spec(k)/G]$)
to its coarse moduli space, as in Lemma \ref{lem:r_j}.

The morphisms $i_*: \Gt_j(k) \to \Gt_j(\A^n/G)$ are zero as they factor 
through the push-forward $\Gt_j(k) \to \Gt_j(\A^n)$ which are
zero maps by the Bass formula
in $\Gt$-theory \cite[chapter 6 Theorem 8. ii]{quillen}.
Thus the localization sequence for $\Gt$-theory of $\A^n/G$ splits into short exact sequences.

Using isomorphisms given by Lemma \ref{lem:r_j}, from the commutative ladder \eqref{eq:ladder}
for every $j \ge 1$ we obtain the diagram:
\[\xymatrix{ 
0 \ar[r] & \Coker(r_j) \ar[r] \ar[d]_{\beta_j}  & \Kt_j^G(\mathbb{A}^n\setminus \{0\}) \ar[r] \ar[d] & \Ker(r_{j-1}) \ar[r] \ar[d]_{\alpha_{j-1}}  & 0 \\
0 \ar[r] & \Gt_j(\mathbb{A}^n/G) \ar[r] & \Gt_j((\mathbb{A}^n\setminus \{0\})/G) \ar[r] & \Kt_{j-1}(k) \ar[r] & 0
}\]

Since $0$ is the only fixed point of the action, the action of $G$
on $\A^n \setminus \{0\}$ is free,
so that the middle vertical map is an isomorphism and using the Snake Lemma we deduce 
an isomorphism
\[
\Ker(\alpha_{j-1}) \simeq \Coker(\beta_j). 
\]

From Lemma \ref{lem:r_j} (2) we get $\Ker(\alpha_j) \simeq T_j$
and from Lemma \ref{lem:r_j} (3) we get $\Coker(\beta_j) \simeq \Coker(p^*: \Kt_j(k) \to \Gt_j(\A^n/G))$.
Since $p^*$ is injective (it is split by any smooth point $x_1 \in \A^n/G$), we obtain the 
exact sequence \eqref{eq:G_j-AnG}.

(3) Since $\A^n/G$ is contractible by Lemma \ref{lem:contractible},
one gets
\eqref{eq:K_jsg-AnG}
by plugging in \eqref{eq:G_j-AnG}
into the exact sequence of Proposition \ref{prop:A1-Ksg} (3).
\end{proof}

\subsection{The global case}

\begin{theorem}\label{theorem:main-thm}
Let $k$ be an algebraically closed field of characteristic zero and let
$X$ be an $n$-dimensional quasi-projective variety.
Assume that $X$ has only isolated quotient singularities $x_1, \dots, x_m$ with isotropy groups $G_1, \dots, G_m$,
i.e. the completions $\wh{\OO}_{X,{x_i}}$ are isomorphic to $\wh{\OO}_{\A^n/G_i,{0}}$
where each $G_i \subset \GL_n(k)$ is a finite group acting freely away from the origin.
Then
\begin{itemize}
\item[(1)] $\Ksg_0(X) \subset \KKsg_0(X)$ are finite abelian groups, 
annihilated by $\lcm(|G_1|,\ldots ,|G_m|)^{n-1}.$
\item[(2)] $\Ksg_1(X)=0$.
\item[(3)] For all $j \ge 1$, we have $\Ksg_{-j}(X)=0$.
\end{itemize}

In addition, if $\dim(X)=2$, then $\KKsg_0(X) \simeq \widehat{G_1}\times\ldots\times\widehat{G_m}$.
\end{theorem}

\begin{proof}
%

By Orlov's Completion Theorem \ref{theorem:idemp-formal-completion} 
and Corollary \ref{cor:A^n/G-idempotent-complete} we obtain equivalences
\[
\overline{\Dsg(X)}\simeq
\bigoplus_{i=1}^m \overline{\Dsg(\A^n/G_i)} \simeq
\bigoplus_{i=1}^m \Dsg(\A^n/G_i).
\]
induced by functors between dg-enhancements.

Thus by definition of the singularity $\Kt$-theory groups and Remark \ref{rem:KK}
we have
\[
\Ksg_0(X) \subset \KKsg_0(X) \simeq \bigoplus_{i=1}^m \Ksg_0(\A^n/G_i)
\]
and for $j \neq 0$
\[
\Ksg_j(X) \simeq \bigoplus_{i=1}^m \Ksg_j(\mathbb{A}^n/G_i).
\]

Now (1) follows from Propositions \ref{prop:G_0}, \ref{prop:K_0^Sg-and-order},
(2) follows from Proposition \ref{K_1^Sg=0} and
(3) follows from Proposition \ref{prop:K_-1}.

Finally if $\dim(X) = 2$, we have isomorphisms
$
\Ksg_0(\A^n/G_i) = \wh{G_i}
$
by Proposition \ref{prop:AnG-c1} ($G_i$ acts freely on $\A^n \setminus \{0\}$ and in
particular has no reflections) so that in this case 
$\KKsg_0(X) \simeq \wh{G_1} \times \dots \times \wh{G_m}$.
\end{proof}

\begin{corollary}\label{cor:main-corollary} 
Under the assumptions of Theorem \ref{theorem:main-thm} the following is true.
\begin{itemize}
\item[(1)] We have an exact sequence $0\to \Kt_0(X)\to \Gt_0(X)\to \Ksg_0(X)\to 0$.
\item[(2)] $\Kt_{-1}(X)$ is a finite torsion abelian group satisfying the same condition on orders as $\Ksg_0(X)$ (see Theorem \ref{theorem:main-thm} (1)).
\item[(3)] For all $j \ge 2$, we have $\Kt_{-j}(X)=0$.
\end{itemize}
\end{corollary}

\begin{proof}
This follows from Theorem \ref{theorem:main-thm} and 
Lemma \ref{lem:sing-K-th}.
\end{proof}

\begin{remark}
The injectiviy of the canonical map $\Kt_0(X)\to \Gt_0(X)$ will generally fail if either:
\begin{itemize}
 \item[(a)] $X$ has quotient singularities which are not isolated, see Example \ref{ex:toric}
 \item[(b)] $X$ has an isolated rational singularity which is not a quotient singularity, see Example \ref{ex:cubic-cone}
\end{itemize}
\end{remark}

\begin{remark}
We do not know if $\Ksg_0(X) = \Coker(\Kt_0(X) \to \Gt_0(X))$ is torsion for any variety $X$ with quotient singularities, not necessarily isolated ones.
The result is known to be true for simplicial toric varieties \cite{BV}.
\end{remark}

\begin{example}
One of the simplest examples of a projective surface $X$ with quotient singularities and 
non-vanishing $\Kt_{-1}(X)$ is the following one.
Consider $G = \Z_2$ acting on $\P^1$ via $[x:y] \mapsto [x:-y]$ and let 
$X = (\P^1 \times \P^1) / \Z_2$ where the action is diagonal. Thus $X$ has four ordinary double points
as singularities.
Using \cite{weibel-surfaces} one can compute that $\Kt_{-1}(X) = \Z_2$.
\end{example}

\subsection{Relation to the resolution of singularities} 

Let $\pi: Y \to X$ be a resolution of singularities. Here $X$ is a variety and $Y$ is a variety or more generally a Deligne-Mumford stack.
If we assume that singularities of $X$ are rational which by definition means that $\pi_* \OO_Y = \OO_X$, 
using the projection formula we get a commutative diagram 
\[
\xymatrix{
 & \Db(Y) \ar@{<-}[ld]_{\pi^*}\ar@{->}[rd]^{\pi_*} & \\
 \Perf(X)\ar@{->}[rr]^{} && \Db(X)}
\]
where the functors $\Perf(X) \to \Db(X)$ and $\pi^*$ are both fully faithful.

We also get an induced diagram on the Grothendieck groups
\begin{equation}\label{diagram:res-of-sing}
\xymatrix{
 & \Kt_0(Y) \ar@{<-}[ld]_{\pi^*}\ar@{->}[rd]^{\pi_*} & \\
 \Kt_0(X)\ar@{->}[rr]^{\can} && \Gt_0(X)}
\end{equation}

\begin{theorem}\label{thm:res-of-sing-injective}
If $X$ is a quasi-projective variety over an algebraically closed field of characteristic zero and with only isolated quotient singularities, then $\pi^*: \Kt_0(X) \to \Kt_0(Y)$ is injective.
\end{theorem}

\begin{proof}
By Corollary \ref{cor:main-corollary}, $\Kt_0(X) \overset{\can}{\to} \Gt_0(X)$ is injective. Injectivity of $\pi^*$ follows from the diagram (\ref{diagram:res-of-sing}).
\end{proof}

\begin{remark}
The injectivity of $\pi^*: \Kt_0(X) \to \Kt_0(Y)$ does not follow from the fact that $\pi^*: \Perf(X) \to \Db(Y)$ is fully faithful and will generally fail for rational singularities.
Indeed in Examples \ref{ex:toric}, \ref{ex:cubic-cone} varieties with rational singularities have huge $\Kt_0(X)$, but admit resolutions with finitely generated $\Kt_0(Y)$.

In dimension up to three, Theorem \ref{thm:res-of-sing-injective} has been known since the work of Levine 
\cite[Corollary 3.4]{levine}
and for normal
surfaces with rational singularities an analogous result follows from the work of Krishna and 
Srinivas \cite[Corollary 1.5]{Krishna-Srinivas}.
\end{remark}

There is an apparent duality between $\pi_*$ and $\pi^*$ in the diagram (\ref{diagram:res-of-sing}).
Instead of injectivity of $\pi^*$ we can ask about surjectivity of $\pi_*$, which indeed sometimes holds.

\begin{theorem}\label{thm:res-of-sing-ess-surjective} 
Let $X$ be a variety over a field $k$ characteristic zero with quotient singularities (not necessarily isolated) 
and let $\pi :Y\to X$ be a resolution of singularities, where $Y$ is a variety or more generally a Deligne-Mumford stack. Then:
\begin{itemize}
\item[(1)] The pushforward $\pi_* :\Db(Y)\to \Db(X)$ is essentially surjective.
\item[(2)] The pushforward induces an exact equivalence 
\[
\Db(Y)/\ker(\pi_*)\xrightarrow{\simeq}\Db(X). 
\]
\end{itemize}
In particular $\pi_*: \Kt_0(Y) \to \Gt_0(X)$ is surjective.
\end{theorem}

\begin{lemma}\label{lem:single-resolution}
If $k$ is a field of characteristic zero and the statement (1) or (2) of Theorem \ref{thm:res-of-sing-ess-surjective} holds for a single resolution $\pi: Y \to X$, then it holds for all resolutions of $X$.
\end{lemma}
\begin{proof}
The proof is a standard application of the Weak Factorization Theorem \cite{Wlod, AKMW}, 
extended to Deligne-Mumford stacks in \cite{bergh1}. Given a birational isomorphism between Deligne-Mumford orbifolds, that is 
Deligne-Mumford stacks with trivial generic stabilizers, it can be decomposed into a sequence of stacky blows ups and blow downs with smooth centers. 
This means that given two resolutions $\pi: Y \to X$, $\pi': Y' \to X$ we may assume that $Y'$ is obtained from $Y$ by a single smooth stacky blow up
$\gamma: Y' \to Y$.
Recall that by definition a stacky blow up is either a blow up of a substack, or a root stack along 
a smooth divisor, and in each case we have $\gamma_* \OO_{Y'} \simeq \OO_Y$ (see e.g. \cite[Example 4.6]{BLS}).

We get a commutative diagram
\[\xymatrix{
\Db(Y') \ar[dr]_{\pi'_*} \ar[rr]^{\gamma_*} && \Db(Y) \ar[dl]^{\pi_*} \\ 
& \Db(X) &
}\]

Furthermore the adjoint pair $\gamma^*$, $\gamma_*$ satisfies $\gamma_* \gamma^* = \id$ so that $\gamma^*$ is fully-faithful and there is a semi-orthogonal decomposition
\[
\Db(Y') = \langle \Ker(\gamma_*), \gamma^* \Db(Y) \rangle. 
\]

In particular $\gamma_*$ is essentially surjective and condition (1) for $Y$ is equivalent to condition (1) for $Y'$.
Furthermore we have a semi-orthogonal decomposition 
\[
\Ker(\pi'_*) = \langle  \Ker(\gamma_*), \gamma^* \Ker(\pi_*) \rangle 
\]
which induces an equivalence of Verdier localizations
\[\xymatrix{
\Db(Y')/\Ker(\pi'_*) \ar[dr]_{\pi'_*} \ar[rr]^{\simeq} && \Db(Y)/\Ker(\pi_* ) \ar[dl]^{\pi_*} \\ 
& \Db(X) &
}\]
so that conditions (2) for $Y$ and $Y'$ are equivalent as well.
\end{proof}

\begin{lemma}\label{lem:t-structure}
If $\pi: Y \to X$ is a resolution of rational singularities and $\Dm(Y)$ admits a $t$-structure which induces a bounded $t$-structure on $\Db(Y)$ and
for which $ \pi_*: \Dm(Y) \to \Dm(X)$ is $t$-exact,
then (1) and (2) of Theorem \ref{thm:res-of-sing-ess-surjective} are true.
\end{lemma}

\begin{proof} 

We temporarily use the notation $\Ker^b(\pi_*):=\Ker(\pi_*)\cap\Db(Y)$. 
We will show that the functor $\ol{ \pi_*}: \Db(Y)/\Ker^b( \pi_*) \to \Db(X)$ is essentially surjective 
and fully faithful.

Essential surjectivity is proved in the same way as in \cite[Corollary 2.5]{Kuznetsov-sextics}.
For every $\EE \in \Db(X)$ and $N \ge 1$ we consider the distinguished triangle
\[
 \pi^* \EE \to\tau^{\ge -N}_\AA   \pi^* \EE \to C ,
\]
where $\tau^{\ge -N}_\AA$ denotes the truncation with respect to the corresponding $t$-structure $\AA$ on $\Dm(Y)$. We apply $\pi_*$ to this triangle. Since $ \pi_*  \pi^* = \id$ and $ \pi_*$ is $t$-exact, in particular $\pi_*$ 
commutes with truncation functors,
the push-forward of the triangle above has the form:
\[
 \EE \to\tau^{\ge -N} \EE \to  \pi_* C,
\]
where $\tau^{\ge -N}$ is the truncation with respect to the standard
$t$-structure on $\Dm(X).$
Since $\EE$ is bounded, for sufficiently large $N$ we have $ \pi_* C = 0$ so that $C \in \Ker( \pi_*)$. 
In particular, $\ol{\pi_*}: \Db(Y)/\Ker^b( \pi_*) \to \Db(X)$ is essentially surjective. 

On the other hand, one observes by the diagram
\[
\xymatrix{
\Dm(Y)/\Ker(\pi_*)\ar[r]^-{\ol{\pi_*}}_-{\simeq} & \Dm(X) \\ 
\Db(Y)/\Ker^b(\pi_*) \ar[u] \ar[r]^-{\ol{\pi_*}} & \Db(X)\ar@{^{(}->}[u]  \\
}
\]
that $\ol{\pi_*}$ is fully faithful if and only if the natural functor 
\[
\Db(Y)/\Ker^b(\pi_*)\to \Dm(Y)/\Ker(\pi_*)
\]
is fully faithful. To show this, we use Verdier's criterion 
\cite[Theorem 2.4.2]{verdier}. 
We see that if $C\to B$ is a morphism in $\Dm(Y)$ with $C\in\Ker(\pi_*)$ and $B\in\Db(Y)$, then for a big enough $N$ (depending on $B$) this morphism factors through $\tau^{\ge -N}_\AA C$. 
Since $\pi_*$ commutes with $\tau^{\ge -N}_\AA$, one easily sees that $\tau^{\ge -N}_\AA C\in\Ker^b(\pi_*)$.
\end{proof}


\begin{proof}[Proof of Theorem \ref{thm:res-of-sing-ess-surjective}]
By Lemma \ref{lem:single-resolution} it suffices 
to check the statement for a single resolution. We consider the canonical stack 
$\pi: \XX_{can} \to X$ over $X$ \cite[Remark 4.9]{FMN}.
Since we assume $k$ has characteristic zero,
the pushforward $ \pi_*: \XX_{can} \to X$ is exact. The proof is finished using Lemma \ref{lem:t-structure}.
\end{proof}

\begin{remark}
Statements (1) and (2) of Theorem \ref{thm:res-of-sing-ess-surjective} for a resolution of arbitrary rational singularities $\pi: Y \to X$ is an old open question going back to Bondal and Orlov \cite{bondal-orlov-icm}. 
In addition to quotient singularities the answer is positive in the case of cones over smooth Fano varieties \cite{Efimov}, and for rational singularities such that fibers of a resolution 
$Y \to X$ have dimension at most one \cite{Kuznetsov-sextics} (in \cite[Corollary 2.5]{Kuznetsov-sextics}
property (1) is proved, while property (2) follows from Lemma \ref{lem:t-structure}).
\end{remark}

\section{Examples and Applications}

In this section $k$ is an algebraically closed field of characteristic zero.


\subsection{Torsion-free $\Kt_0(X)$}

\begin{application}[\cite{levine0, levine}]
Let $X$ be a projective rational surface with isolated quotient singularities. Then $\Kt_0(X)$ is a free abelian group of the same rank as $\Gt_0(X)$.

Indeed if $\pi: Y \to X$ is a resolution, then by Theorem \ref{thm:res-of-sing-injective} we have an injection $\pi^*: \Kt_0(X) \to \Kt_0(Y)$. Since $Y$ is a smooth projective rational surface,
$\Kt_0(Y)$ is a free abelian group of finite rank, and the same is true for $\Kt_0(X)$. Finally by Corollary \ref{cor:main-corollary}, we have an inclusion $\Kt_0(X) \subset \Gt_0(X)$ and the ranks of the two groups are equal. 

Note that $\Gt_0(X)$ will typically have non-zero torsion.
\end{application}

\begin{application}[Weighted projective spaces with coprime weights]\label{appl:weighted-Pn}
Let $X = \P(a_0, \dots, a_n)$ be a weighted projective space. 
Let us assume that the weights $a_0, \dots, a_n$ are pairwise coprime. In this case singularities of $X$ are isolated,
and using our results we show that $\Kt_0(X)$ is a free abelian group of rank $n+1$. 

Indeed if we let $Y = [\P(a_0, \dots, a_n)]$ to be the weighted projective stack, the natural morphism $\pi: Y \to X$ is a resolution of singularities,
and by Theorem \ref{thm:res-of-sing-injective}, $\pi^*$ is injective. Since $\Kt_0(Y)$ is a free abelian of finite rank \cite[Theorem 5.6]{Joshua-Krishna}, the same is true for $\Kt_0(X)$.
To compute the rank of $\Kt_0(X)$ we can use the following argument: by Corollary \ref{cor:main-corollary} we have an isomorphism $\Kt_0(X) \otimes \Q \simeq \Gt_0(X) \otimes \Q$
and the latter space is $(n+1)$-dimensional, which can be seen by comparing $\Gt_0(X)$ to Chow groups \cite{AlAmrani}.
Thus we get $\Kt_0(X) \simeq \Z^{n+1}$.
\end{application}

\subsection{$ADE$ curves and threefolds}

We consider one-dimensional $ADE$ singularities over an algebraically closed field of characteristic zero.
For each such curve $C$ we compute $\Pic(C)$ as well as $\Ksg_0(C)$ and $\Ksg_1(C)$.
If $N$ is the number of irreducible components of the curve, then $\Ksg_0(C)$ is computed using Corollary \ref{cor:K_0^Sg-of-curves-and-surfaces}.
We record the results in the table:

\medskip

\begin{equation}\label{tab:ADE-1}
\begin{tabular}{ |c|c|c|c|c|c| }
\hline
$C$ & Equation & $N$ & $\Ksg_0(C)$ & $\Pic(C)$ & $\Ksg_1(C)$ \\
\hline
 $A_{2l}$, $l \geq 1$ & $y^2+z^{2l+1}$ & $1$ & $0$ & $k^l$ & $k^l$ \\ 
\hline
 $A_{2l-1}$, $l \geq 1$ & $y^2+z^{2l}$ & $2$ & $\Z$  & $0$ & $k^* \oplus \Z$ \\ 
\hline
 $D_{2l}$, $l\geq 2$ & $y^2z+z^{2l-1}$ & $3$ & $\Z^2$  & $0$ & $(k^* \oplus \Z)^2$ \\  
\hline
 $D_{2l-1}$, $l \geq 3$ & $y^2z+z^{2l-2}$ & $2$ & $\Z$  & $k^{l-2}$ & $[k^* \oplus \Z; k^{l-2}]$ \\  
\hline
 $E_6$ & $y^3+z^4$ & $1$ & $0$  & $k^3$ & $k^3$ \\
\hline
 $E_7$ & $y^3+yz^3$ & $2$ & $\Z$  & $k$ & $[k^* \oplus \Z; k]$ \\
\hline
 $E_8$ & $y^3+z^5$ & $1$ & $0$  & $k^4$ & $k^4$ \\
\hline
\end{tabular}
\end{equation}

\medskip

We use the notation $[A; B]$ to denote an abelian group which has a subgroup $A$ with quotient $B$.
The first singularity $\Kt$-theory groups $\Ksg_1(C)$ are computed using the following Proposition.

\begin{proposition}\label{prop:ade-curves-K1}
For every 
$ADE$ singularity we have a natural exact sequence
\[
0\to (k^* \oplus\Z)^{N-1}\to \Ksg_1(C)\to\Pic(C)\to 0.
\]
\end{proposition}
\begin{proof}
By Proposition \ref{prop:A1-Ksg} (3) and using the fact that 
$\Kt_0(C) = \Z \oplus \Pic(C)$
(see \cite[Remark 1 on page 126]{fulton-lang})
we get a short exact sequence
$$0\to \Gt_1(C)/k^*\to \Ksg_1(C)\to \Pic(C)\to 0.$$
To finish the proof we show that $\Gt_1(C) = (k^*)^N \oplus \Z^{N-1}$, and the morphism $\Kt_1(k) \to \Gt_1(C)$ maps $k^*$ into $(k^*)^{N}$ diagonally.
This is done using the localization sequence for the closed embedding $i: \{0\} \to C$.
Since for every $j \ge 0$, $i_*: \Gt_j(k) \to \Gt_j(C)$ factors through any component $\A^1$ of the normalization of $C$, it is a zero map,
and we get a short exact sequence
\[
0 \to \Gt_1(C)\to \Gt_1(\A^1 \setminus \{0\})^N \to \Gt_0(k) \to 0
\]
which finishes the proof as $\Gt_1(\A^1 \setminus \{0\}) = k^* \oplus \Z$.
\end{proof}

\begin{lemma}\label{lem:cusp-curves}
If $C$ is a curve with equation $x^a - y^b = 0$ where $\gcd(a,b) = 1$,
then we have an isomorphism 
\[
\Pic(C) \simeq k^{\frac12 (a-1)(b-1)}.
\]
\end{lemma}
\begin{proof}
Consider $\pi: \A^1 \to C$ given by $\pi(t) = (t^b, t^a)$.
Under the condition $\gcd(a,b) = 1$, $\pi$ is surjective which implies irreducibility of $C$,
and since $\pi$ is finite of degree one, $\pi$ is the normalization morphism.
By \cite[Corollary 3.3]{Oort} we get an isomorphism of abelian groups
\[
\Pic(C) \simeq k[t] / k[t^a, t^b].
\]
Thus $\Pic(C)$ obtains a $k$-vector space structure with a $k$-basis 
corresponding of $t^i$, for every $i \ge 0$ which can not 
be represented as a non-negative integer combination of $a$ and $b$.

By a classical theorem of Sylvester, the number of positive integers not representable by non-negative integer combinations of $a$ and $b$
is equal to $\frac12 (a-1)(b-1)$ (see \cite{NW} for a modern treatment)
so that we have an isomorphism of abelian groups
\[
k[t] / k[t^a, t^b] \simeq k^{\frac12 (a-1)(b-1)}.
\]
\end{proof}

Every $ADE$ curve $C$ is a union of components isomorphic to $\A^1$ and at most one component $C_0$ with equation $x^a - y^b = 0$.
Trivializing line bundles on each affine line component, we deduce that $\Pic(C) = \Pic(C_0)$.
Proposition \ref{prop:ade-curves-K1} and Lemma \ref{lem:cusp-curves} allow us to fill in the table 
\eqref{tab:ADE-1}.

\medskip

We demonstrate what the singularity $\Kt$-theory has to do with the question of computing class groups. 
The applications below can be obtained by other methods too, however, we demonstrate the approach 
which relies on Kn\"orrer periodicity shifting the topological filtration (Proposition \ref{prop:knorrer-shift}).

%

\begin{application}\label{appl:Knorrer-3}
Let $C \subset \A^2$ with coordinates $z,w$ be given by $g(z,w) = 0$ and let
$X \subset \A^4$ with coordinates $x,y,z,w$ be given by $xy + g(z,w) = 0$. 

Let us assume that $C$ is reduced.
Since we have $\Sing(X) = \{ (0, 0) \} \times \Sing(C)$, the latter condition is equivalent
to $X$ having isolated singularities, and since $X$ is a hypersurface, it
is irreducible and normal.

Let $N$ be the number of irreducible components of $C$.
By Proposition \ref{prop:knorrer-shift} and Proposition \ref{prop:K_0sg-filtr} we have an isomorphism
\[
\Cl(X) / \Pic(X) = \Gr^1 \Ksg_0(X) \simeq \Gr^0 \Ksg_0(C) = \Z^{N-1},
\]
in particular $X$ is factorial if and only if $C$ is irreducible.
\end{application}

\begin{example}\label{ex:ADE-dim3}
We can compute the class group of the standard forms of three-dimensional 
$ADE$ singularities. Since $X$ is given by a weighted homogeneous equation, 
$\Pic(X) = 0$ \cite[Lemma 5.1]{Murthy}
we have $\Cl(X) \simeq \Ksg_0(X) = \Z^{N-1}$.
We put the results in the table (cf. table \ref{tab:ADE-1}):
\begin{equation}\label{tab:ADE-3}
\begin{tabular}{| c | c | c |}
\hline
 $X$ & Equation & $\Cl(X)$ \\ 
\hline
 $A_{2k} \; (k \ge 1)$ & $xy + z^2+w^{2k+1}$ & $0$ \\ 
\hline
 $A_{2k-1} \; (k \ge 1)$ & $xy + z^2+w^{2k}$ & $\Z$ \\ 
\hline
 $D_{2k} \; (k \ge 2)$ & $xy + z^2w+w^{2k-1}$ & $\Z^2$ \\  
\hline
 $D_{2k-1} \; (k \ge 3)$ & $xy + z^2w+w^{2k-2}$ & $\Z$ \\  
\hline
 $E_6$ & $xy + z^3+w^4$ & $0$\\
\hline
 $E_7$ & $xy + z^3+zw^3$ & $\Z$ \\
\hline
 $E_8$ & $xy + z^3+w^5$ & $0$\\
\hline
\end{tabular}
\end{equation}
\end{example}

\subsection{Non-vanishing $\Ksg_1(X)$}

In this section we collect some examples where $\Ksg_1(X)$ is nonzero.
From the singularity $\Kt$-theory exact sequence \eqref{eq:seq-k} it follows that $\Ksg_1(X)$
surjects onto $\Ker(\Kt_0(X) \overset{\can}{\to} \Gt_0(X))$.

\begin{example}[Non-isolated quotient singularity with huge kernel $\Ker(\Kt_0(X) \to \Gt_0(X)$]\label{ex:toric}
The first such example has been constructed by Gubeladze \cite{gubeladze}.
We present an example given by Corti\~nas, Haesemeyer, Walker and Weibel \cite[Example 5.10]{weibel-co-toric}.

Let $E = \OO \oplus \OO(2)$ be the rank two bundle over $\P^1$.
Let $\Z_2$ act on $E$ fiberwise via $v \mapsto -v$. 
Then $X = E / \Z_2$, has quotient singularities and its singular locus isomorphic to $\P^1$.

The canonical map $\Kt_0(X)\to \Gt_0(X)$ is not injective, and furthermore, the kernel $\Ker(\Kt_0(X) \to \Gt_0(X))$ is huge,
that is contains the base field $k$ as a subgroup.
\end{example}

\begin{example}[Isolated rational singularity with huge kernel $\Ker(\Kt_0(X) \to \Gt_0(X)$] \label{ex:cubic-cone}
Consider a smooth cubic hypersurface $S \subset \P^3$, 
and let $X \subset \A^4$ be the affine cone over $S$. 
Then $X$ has an isolated rational singularity. 

The Grothendieck group of a cone over a smooth variety has been computed in \cite{weibel-co}. In particular since $\chi(\TT_S) = -4$ by Riemann-Roch so that 
$H^1(S,\Omega^1_{S/ \Q}(1))=H^1(S, \Omega^1_{S/k}(1)) = H^1(S, \TT_S) \ne 0$, where the first equation follows from the short exact sequence
$$0\to \Omega^1_{k/ \Q}(1)\to\Omega^1_{S/ \Q}(1)\to\Omega^1_{S/k}(1)\to 0,$$
see \cite[Proposition 20.6.2]{EGA4}. The main result of \cite{weibel-co} implies that $\Kt_0(X)$ is huge, that is it contains a nonzero $k$-vector space.

Finally by Proposition \ref{prop:A1-Ksg} (1)
the canonical map $\Kt_0(X)\to \Gt_0(X)$ factors through $\Z$, so that $\Ker(\Kt_0(X) \to \Gt_0(X))$ 
is huge as well.
%
\end{example}

\begin{example}[Non-vanishing $\Ksg_1(X)$ for isolated quotient singularities over non-algebraically closed fields]\label{ex:A1-any-field}
Let $X = \A^2/\Z_2$ be the quotient by the action $v \mapsto -v$. 
We claim that $\Ksg_1(X) \simeq k^* / (k^*)^2$.

Indeed, $X$ is isomorphic to the affine surface $xy + z^2 = 0$, and using 
the Kn\"orrer periodicity Theorem \ref{Thm:Knorrer} 
we have
\[
\Ksg_1(X) \simeq \Ksg_1(R), 
\]
where $R = k[\epsilon]/(\epsilon^2)$. We compute the singularity $\Kt$-theory via the $\Kt$-theory sequence \eqref{eq:seq-k}, 
plugging in $\Gt_i(R) = \Kt_i(k)$, as $\Gt$-theory is independent of the non-reduced scheme structure:
\[
\Kt_1(R) \to k^* \to \Ksg_1(R) \to \Kt_0(R) \to \Z \to \Ksg_0(R) \to 0.
\]
Now $\Kt_0(R) = \Z$ and the map $\Z = \Kt_0(R) \to \Z$ is multiplication by two (cf Example \ref{ex:k-eps}),
and similarly $\Kt_1(R) = R^*$, and the map $\Kt_1(R) \to k^*$ is $a + b\epsilon \mapsto a^2$ \cite[Example 10.2]{Holm}.
We get
\[
\Ksg_1(X) \simeq \Ksg_1(R) \simeq k^* / (k^*)^2,
\]
which is in general a non-finitely generated $2$-torsion group.
\end{example}

\subsection{Proof of a conjecture of Srinivas for quotient singularities}

In the 1980s Srinivas considered the question 
whether for an isolated quotient singularity $x_0 \in X$
the length homomorphism $l: \Kt_0(X \text{ on } x_0) \to \Z$ is an isomorphism 
\cite[Page 38]{srinivas2}.
Here $\Kt_0(X\text{ on } x_0)$ stands for the Grothendieck group of perfect 
complexes supported at $x_0$ (originally Srinivas has considered the Grothendieck group
of coherent sheaves which are supported at the singular points and which are perfect as complexes, but 
by \cite[Proposition 2]{RS} these two groups are isomorphic).

Levine has proved that $l$ is an isomorphism if $X$ is two-dimensional with isolated 
quotient singularities
\cite[Theorem 3.2]{levine0}, 
that $l$ is always surjective for isolated Cohen-Macaulay singularities, 
and that it has torsion kernel \cite[Proposition 2.6, Theorem 2.7]{levine} 
in the case of isolated quotient-singularities.

The language of the singularity $\Kt$-theory is well-adapted to deal with this kind of questions. 

\begin{lemma}\label{lem:Srinivas-general}
Let $k$ be an algebraically closed field. 
Let $X$ be a quasi-projective variety with isolated singularities. 
There is an exact sequence 
\begin{equation}\label{eq:diag-srinivas}
\Ksg_1(X) \to \Kt_0(X\text{\;\rm on }\Sing(X)) \overset{l}{\to} \Z^{\Sing(X)} \to 0
\end{equation}
and a natural surjective homomorphism $\Ker(l) \to \Ker(\can: \Kt_0(X) \to \Gt_0(X))$.
\end{lemma}
\begin{proof}
We consider the diagram of pretriangulated dg-categories
\[\xymatrix{
\Dperf_{dg}(X \text{ on } \Sing(X)) \ar[d]\ar[r] & \Db_{dg}(X \text{ on } \Sing(X)) \ar[d] \ar[r] & \Dsg_{dg}(X \text{ on } \Sing(X)) \ar[d]  \\
\Dperf_{dg}(X) \ar[r] & \Db_{dg}(X) \ar[r] & \Dsg_{dg}(X)   \\
}\]
and the associated long exact sequences of Schlichting's $\Kt$-groups:
\[\xymatrix{
\Ksg_1(X \text{ on } \Sing(X)) \ar[r] \ar[d]_{\simeq} & \Kt_0(X \text{ on } \Sing(X)) \ar[r] \ar[d] & 
\Gt_0(X \text{ on } \Sing(X)) \ar[r] \ar[d] & \Ksg_0(X \text{ on } \Sing(X))  \ar@{^{(}->}[d] \\
\Ksg_1(X) \ar[r] & \Kt_0(X) \ar[r]^{\can} & \Gt_0(X) \ar[r] & \Ksg_0(X)   \\
}\]
where the left vertical arrow is an isomorphism and the right vertical arrow
is injective by Lemma \ref{lem:K-th-supp}.

We have $\Ksg_0(X \text{ on } \Sing(X)) = 0$ (Corollary \ref{cor:k-0-supp-sing-zero}),
and we have a natural isomorphism
\[
\Gt_0(X \text{ on } \Sing(X)) = \Gt_0(\Sing(X)) = \Z^{\Sing(X)}
\]
given by the length (dimension) of zero-dimensional coherent sheaves,
so that exact sequence \eqref{eq:diag-srinivas} is the first row in the diagram above.

Finally, the diagram above also induces the surjection $\Ker(l) \to \Ker(\can)$.
\end{proof}

The next result deals with the injectivity part of the Srinivas conjecture for quotient singularities, and thus gives a stronger
version of \cite[Theorem 3.2]{levine0}.

\begin{proposition}\label{prop:Srinivas-quotient}
If $X$ is a quasi-projective variety with isolated quotient singularities 
then the length map
\[
l: \Kt_0(X\text{\;\rm on }\Sing(X))\to \Z^{\Sing(X)}
\]
is an isomorphism.
\end{proposition} 

\begin{proof}
By Theorem \ref{theorem:main-thm} (2), $\Ksg_1(X) = 0$. Lemma \ref{lem:Srinivas-general} implies the result.
\end{proof}

\begin{remark}\label{rem:ker-l}
By Lemma \ref{lem:Srinivas-general} non-vanishing 
of $\Ker(\can: \Kt_0(X) \to \Gt_0(X))$ implies non-vanishing of $\Ker(l)$.
This applies for instance in the case of a cone over a smooth cubic surface, see Example \ref{ex:cubic-cone}.
\end{remark}

\begin{example}
Let $k$ be an arbitrary field with $\chr(k) \ne 2$.
Let $Q_n$ be the $n$-dimensional affine split quadric cone as in Examples \ref{ex:A-ev}, \ref{ex:A-odd}.
It is a result of Levine \cite[Theorem 4.2]{levine} that 
\[
\Kt_0(Q_n \text{ on } 0) \simeq \left\{\begin{array}{ll}
                                  \Z \oplus k^*/(k^*)^2 , & n \text{ even } \\
                                  \Z^2 \oplus k^* , & n \text{ odd } \\
                                  \end{array}\right.
\]
This result can be reproved using exact sequence \eqref{eq:diag-srinivas}
and the fact that 
\[
\Ksg_1(Q_n) \simeq \left\{\begin{array}{ll}
                                  k^*/(k^*)^2 , & n \text{ even (cf Example \ref{ex:A1-any-field})} \\
                                  \Z \oplus k^* , & n \text{ odd (cf $A_1$ case in \eqref{tab:ADE-1})} \\
                                  \end{array}\right.
\]
Similarly, one can compute $K_0(X \text{ on } 0)$ for other ADE singularities of arbitrary dimension.
We omit the details of this computation.
\end{example}

\providecommand{\arxiv}[1]{{\tt{arXiv:#1}}}

\end{document}